\documentclass[%
]{mpi2015-cscpreprint}

\usepackage[american]{babel}

\usepackage{graphicx}

\usepackage{amssymb}
\usepackage{amsthm}
\usepackage{mymacros}
\usepackage{array}

\usepackage{amsmath}
\usepackage{amsfonts}
\usepackage{amssymb}
\usepackage{amsthm}

\usepackage{booktabs}

\usepackage{placeins}

\usepackage{algorithm}
\usepackage{algorithmicx}
\usepackage{algpseudocode}

\usepackage{subcaption}
 \usepackage{numprint}
 \usepackage{multirow}

\usepackage{pgfplots}
\pgfplotsset{compat=1.13}
\usetikzlibrary{shapes.multipart,decorations.markings,decorations.text, 
decorations.pathreplacing,patterns,positioning,intersections,calc,matrix}

\newtheorem{theorem}{Theorem}
\newtheorem{proposition}[theorem]{Proposition}
\newtheorem{lemma}[theorem]{Lemma}

\newtheorem{definition}{Definition}

\begin{document}
  

\title{A Structure-Preserving Divide-and-Conquer Method for Pseudosymmetric Matrices}
  
\author[$\ast$]{Peter Benner}
\affil[$\ast$]{Max Planck Institute for Dynamics of Complex Technical Systems, \authorcr 
    Sandtorstr. 1, 39106 Magdeburg, Germany.\authorcr
  \email{peter.benner@mpi-magdeburg.mpg.de}, \orcid{0000-0003-3362-4103}}
\author[$\dagger$]{Yuji Nakatsukasa}
\affil[$\dagger$]{Mathematical Institute, University of Oxford, \authorcr
    Oxford, OX2 6GG, UK.
        \authorcr
  \email{nakatsukasa@maths.ox.ac.uk}, \orcid{0000-0001-7911-1501}}
  
\author[$\ddag$]{Carolin Penke}
\affil[$\ddag$]{Max Planck Institute for Dynamics of Complex Technical Systems, \authorcr 
    Sandtorstr. 1, 39106 Magdeburg, Germany.    \authorcr
  \email{penke@mpi-magdeburg.mpg.de}, \orcid{0000-0002-4043-3885}}
  
\shortauthor{P. Benner, Y. Nakatsukasa, C.Penke}
  

  
\abstract{
 We devise a spectral divide-and-conquer scheme for matrices that are self-adjoint with respect to a given indefinite scalar product (i.e.\ \emph{pseudosymmetic} matrices). The pseudosymmetric structure of the matrix is preserved in the spectral division, such that 
the method can be applied recursively to achieve full diagonalization. The method is well-suited for structured matrices  that come up in computational quantum physics and chemistry. In this application context, additional definiteness properties guarantee a convergence of the matrix sign function iteration within two steps when Zolotarev functions are used. The steps are easily parallelizable. Furthermore, it is shown that the matrix decouples into symmetric definite eigenvalue problems after just one step of spectral division.
}

\novelty{The spectral divide-and-conquer methodology is extended such that the  structure of pseudosymmetric matrices is preserved. Further results are given regarding the computation of the matrix sign function and subspace bases, in particular when a certain definiteness property holds.}
\maketitle



\section{Introduction}
\label{Sec:Introduction}

Given a diagonalizable matrix $A\in\mathbb{K}^{n\times n}$, where $\mathbb{K}\in\{\mathbb{R},\mathbb{C}\}$, we are interested in full diagonalization, i.e.\ finding $V\in\mathbb{K}^{n\times n}$, such that
\begin{align}\label{Eq:Diagonalization}
 V^{-1} A V = D.
\end{align}
For $\mathbb{K}=\mathbb{C}$, the matrix $D$ is diagonal and contains the eigenvalues of $A$ as diagonal values. For $\mathbb{K}=\mathbb{R}$, $D$ is block diagonal with blocks of size $1\times 1$, corresponding to real eigenvalues, or $2\times 2$, corresponding to a pair of complex conjugate eigenvalues. The well-established standard approach for solving \eqref{Eq:Diagonalization} starts by computing the Schur decomposition of  $A$
\begin{align*}
 Q^{\tran} A Q = T,
\end{align*}
where $Q$ is orthogonal (or unitary) and $T$ is (block) upper triangular, via the QR algorithm \cite{GolV13}. The eigenvectors of $T$ are computed via backward substitution or the eigenvectors of $A$ are recovered via inverse iteration \cite{Lapack}. The QR algorithm, however, has proven difficult to parallelize and is not well-suited for computing only parts of the eigenvalue spectrum \cite{BaiD93}. This is why spectral divide-and-conquer algorithms were explored as an alternative \cite{Mal93,BaiD93,BaiD94,BaiDG97}. They are based on the idea of spectral division. A matrix $V$ is found such that
\begin{align}\label{Eq:SpecDiv}
 V^{-1} A V = \begin{bmatrix}A_{11} & A_{12}\\
               0 & A_{22}
              \end{bmatrix}.
\end{align}
This is achieved when the first columns of $V$ form a basis of an invariant subspace of $A$ and the remaining columns complement them to form a basis of $\mathbb{K}^{n}$. Now, the eigenvalue problems of the smaller matrices $A_{11}$ and $A_{22}$ are considered. Repeating this method recursively leads to a \emph{spectral divide-and-conquer} scheme for the triangularization of a matrix.

The required subspace bases are acquired by employing the matrix sign function, which is computed via an iteration. In general, the operation count of spectral divide-and-conquer methods is higher than that of QR based algorithms. This is why optimized implementations that exploit available parallelism are needed. 

One direction towards more efficient implementations is to take the given structure of a matrix into account. For example, it is clear that symmetry must be exploited when available. In spectral division \eqref{Eq:SpecDiv} exploiting symmetric structure is achieved by finding an orthogonal matrix $V$. This way, a block-diagonalization is realized instead of the block-triangularization. 

This is done in the spectral divide-and-conquer approach presented in \cite{NakH13}. For symmetric matrices, the computation of the matrix sign function can be parallelized particularly well \cite{BenKS14, BenKS14a, NakF16}, making it competitive with standard approaches in a high performance setting \cite{LtaSEetal19, KeyLNetal21}. An important aspect is that spectral divide-and-conquer methods require less communication than QR based approaches. In recent years many efforts have been directed to finding communication-avoiding implementations of essential tools in numerical linear  algebra \cite{BalDHetal11}.  Spectral divide-and-conquer methods can be implemented using these available building blocks \cite{BalDD10}. On more advanced architectures, avoiding communication is more important than avoiding FLOPs in order to minimize the runtime.

In the present work, we extend the spectral divide-and-conquer approach to solve eigenvalue problems of matrices with a more general structure, called pseudosymmetry. 

A pseudosymmetric matrix is symmetric up to sign changes of rows or columns. Symmetric matrices are a subset of pseudosymmetric matrices. The complex analogue is called a pseudo-Hermitian matrix. In the following, statements are formulated for (pseudo-)symmetric matrices, but also hold for (pseudo-)Hermitian matrices.

Efforts to exploit pseudosymmetric structure led to the development of the HR algorithm \cite{Bun79, Bun81, BreG82}. It generalizes the symmetric QR algorithm and is motivated by the following observation. A generalized eigenvalue problem with symmetric matrices
\begin{align}\label{Eq:GenSymEig}
 Ax = \lambda Bx,\qquad A=A^{\tran},\quad B=B^{\tran},
\end{align}
can be cast into a pseudosymmetric standard eigenvalue problem. Neither $A$ nor $B$ need to be positive definite. If $B$ is nonsingular, it has a decomposition $B=R^{\tran}\Sigma R$, where $\Sigma$ is a diagonal matrix with $1$ or $-1$ as diagonal values. Then \eqref{Eq:GenSymEig} is equivalent to
\begin{align*}
 \Sigma R^{-\tran}AR^{-1}y = \lambda y,\qquad y=Rx.
\end{align*}
$\Sigma R^{-\tran}AR^{-1}$ is clearly a pseudosymmetric matrix. Diagonal matrices with $\pm1$ as diagonal values are called \emph{signature matrices} in the following.

The QR algorithm computes its results with great accuracy because only (implicit) orthogonal transformations are involved. This is not true for the HR algorithm, which uses ($\Sigma,\hat{\Sigma}$)-orthogonal matrices instead (also called pseudoorthogonal) \cite{Wat07}. A ($\Sigma,\hat{\Sigma}$)-orthogonal matrix $H$, where $\Sigma$ and $\hat{\Sigma}$ are two signature matrices, fulfills $H^{\tran}\Sigma H = \hat{\Sigma}$. They are used to transform a matrix to upper triangular form, similar to the QR decomposition.

On top of that, the HR algorithm suffers from the same drawbacks as the QR algorithm in a high-performance environment: It is hard to parallelize and not communication-avoiding, as explained above.

The spectral divide-and-conquer method developed in this work presents a promising alternative. It can be parallelized and only relies on building blocks for which communication-avoiding implementations exist. It can be used to compute only parts of the spectrum with reduced computational effort. Stability concerns are addressed by employing alternatives to the HR decomposition in the computation of the matrix sign function, presented in \cite{BenNP22}.

Our main motivation stems from computational quantum science. Time-dependent density functional theory in the linear-response regime (TDDFT) and the Bethe-Salpeter approach are two competing methods for computing excited states of solids or molecules in the context of a perturbed induced density matrix \cite{OniRR02, SagA09}.

The Bethe-Salpeter equation derived from many body perturbation theory \cite{OniRR02,ShaDYetal16,BenKK16} and the Casida equation derived from TDDFT for molecules \cite{Cas95} lead to pseudosymmetric eigenvalue problems in their discretized form: The matrices become symmetric when multiplied with $\Sigma=\diag{I,-I}$, where $I$ denotes the identity matrix.

Due to physical constraints, these matrices have another property, which is exploited in our proposed algorithm. The symmetric matrix resulting from multiplication with $\Sigma$ is positive definite. It was shown in \cite{BenNP22} that for these matrices the proposed iterations have the same favorable convergence properties as in the symmetric setting. Furthermore, we prove in this work that the first round of spectral division decouples the problem into a positive and a negative definite symmetric matrix.  

Pseudosymmetric matrices with these definiteness properties also  play a role in describing damped oscillations of linear systems \cite{Ves11}.

The remainder of the paper is structured as follows. Section \ref{Sec:Prel} introduces scalar products and related notation which form basic concepts used throughout the paper. This refers to a generalization of symmetry and orthogonality with respect to a scalar product defined by signature matrices. $(\Sigma,\hat{\Sigma})$-orthogonal matrices ensure the preservation of structure in the spectral division for pseudosymmetric matrices.  Furthermore, the matrix sign function is introduced as the central tool for spectral division.
Section \ref{Sec:StrucDivConq} explains the idea of spectral divide-and-conquer methods and presents a generalization of this approach for pseudosymmetric matrices. 
The acquisition of $(\Sigma,\hat{\Sigma})$-orthogonal representations of invariant subspaces is essential for structure preservation in the spectral division. In Section \ref{Sec:HypBasis}, we point out a link between QR decompositions of symmetric projection matrices (describing orthogonal projections) and Cholesky factorizations. This link exists analogously for pseudosymmetric projection matrices and the $LDL^{\tran}$ factorization. We use this insight to compute required basis representations via the $LDL^{\tran}$ factorization.
Matrices arising in the application of electronic excitation have additional definiteness properties. Section \ref{Sec:DefPseudosym} shows how these are exploited in the presented algorithms. The computation of proper basis representations simplifies to a partial Cholesky factorization (Section \ref{Subsec:DefPseudoorthRep}) and the computation of the matrix sign function can be accelerated using Zolotarev functions (Section \ref{Subsec:DefZolo}).
Section \ref{Sec:NumRes} presents the results of numerical experiments regarding the new method. 
Conclusions and further research directions are given in Section \ref{Sec:Conclusions}.

\section{Preliminaries}\label{Sec:Prel}
Following \cite{HigMMetal05} and \cite{MacMT05}, we give some basic results regarding non-Euclidian scalar products. A nonsingular matrix $M$ defines a scalar product on $\mathbb{K}^{n}$, where $\mathbb{K}\in\{\mathbb{C},\mathbb{R}\}$, that is a bilinear or sesquilinear form $\langle \cdot, \cdot\rangle_M$,  given by
\begin{align*}
\langle x, y \rangle_M = \begin{cases}
                          x^{\tran} My\text{ for bilinear forms,}\\
                          x^H My\text{ for sesquilinear forms,}
                         \end{cases} 
\end{align*}
for $x,y\in\mathbb{K}^{n}$. We use $\cdot^*$ throughout the paper to indicate transposition 
$\cdot^{\tran}$
or conjugated transposition 
$\cdot^{\herm}$
, depending on whether a bilinear or sesquilinear form is given.

For a matrix $A\in\mathbb{K}^{n\times n}$,  $A^{\star_M}\in\mathbb{K}^{n\times n}$ denotes the adjoint with respect to the scalar product defined by $M$. This is a uniquely defined matrix satisfying the identity
\begin{align*}
 \langle Ax, y \rangle_M = \langle x, A^{\star_M}y \rangle_M
\end{align*}
for all $x,y\in\mathbb{K}^{n}$. We call $A^{\star_M}$ the $M$-adjoint of $A$ and it holds
\begin{align}\label{Eq:StarM}
 A^{\star_M} = M^{-1}A^*M.
\end{align}
%
%
  A matrix $S$ is called ($M$-)self-adjoint (with respect to the scalar product induced by $M$) if $S=S^{\star_M}$.
%

Similar concepts are available for rectangular matrices. As two vector spaces of different dimensions now play a role, two distinct scalar products are considered. We give some clarifying notation following \cite{HigMT10}.
For a matrix $A\in\mathbb{K}^{m\times n}$,  $A^{\star_{M,N}}\in\mathbb{K}^{n\times m}$ denotes the adjoint with respect to the two scalar products defined by the nonsingular matrices $M\in\mathbb{K}^{m\times m}$, $N\in\mathbb{K}^{n\times n}$. This matrix is uniquely defined by the identity
\begin{align*}
 \langle Ax, y \rangle_M = \langle x, A^{\star_{M,N}}y \rangle_N 
\end{align*}
for all $x\in\mathbb{K}^{n},\ y\in\mathbb{K}^{m}$. We call $A^{\star_{M,N}}$ the $(M,N)$-adjoint of $A$ and it holds
\begin{align*}
 A^{\star_{M,N}} = N^{-1}A^*M.
\end{align*}
%
 A matrix $H\in\mathbb{K}^{m\times n}$ is called $(M,N)$-orthogonal when 
 \begin{align*}
  H^{\star_{M,N}}H = I_n.
 \end{align*}
In this case we define
\begin{align}\label{Eq:pseudoinverse}
 H^{\dagger}:=H^{\star_{M,N}} = N^{-1}H^*M.
\end{align}

For $(M,N)$-orthogonal matrices, \eqref{Eq:pseudoinverse} gives the $(M,N)$-Moore-Penrose pseudoinverse discussed in \cite{HigMT10}. This notion generalizes the well-known Moore-Penrose pseudoinverse, which is achieved by setting $M=I_m$ and $N=I_n$.


Our proposed methods rely on the matrix sign function \cite{Rob80,KenL95,Hig08}. Let $A\in\mathbb{K}^{n\times n}$ be a nonsingular matrix with no imaginary eigenvalues with Jordan canonical form
\begin{align*}
 A = Z \begin{bmatrix} J_- &\\
        &J_+
       \end{bmatrix}Z^{-1},
\end{align*}
where $J_-\in\mathbb{K}^{m\times m}$ contains the Jordan blocks associated with the eigenvalues having a negative real part and  $J_+\in\mathbb{K}^{p\times p}$ contains the Jordan blocks associated with the eigenvalues having a positive real part. Then the matrix sign function of $A$ is defined as
\begin{align}\label{eq:signDef}
 \sign{A} := Z \begin{bmatrix} -I_m &\\
        &I_p
       \end{bmatrix}Z^{-1}.
\end{align}

It follows from \eqref{eq:signDef} that the matrix sign function can be used to acquire projectors onto invariant subspaces associated with the positive and negative real parts of the spectrum.
\begin{lemma}\label{Lem:SignProj}
 \item  $P_+ = \frac{1}{2}(I_n+S)$ and $P_-=\frac{1}{2}(I-S)$ are projectors onto the invariant subspaces associated with eigenvalues in the open right and open left half-plane, respectively. 
 \end{lemma}
 In order to acquire projections onto invariant subspaces associated with other eigenvalue subsets, we can use the matrix sign function of a shifted $A+\sigma I$. Another possibility is to transform $A$ before computing the matrix sign function in order to acquire subspaces associated with almost arbitrary regions of the eigenvalue spectrum \cite{BaiDG97}.
   What makes the matrix sign function useful is that there exist iterative methods for its computation \cite{Hig08,KenL91b}. Among the simplest is Newton's iteration to find the roots of $f(x)=x^2-1$,
 \begin{align}\label{Eq:NewtonSign}
  X_{k+1} = \frac{1}{2} (X_k + X_k^{-1}), \quad X_0 = A.
 \end{align}
Our iteration of choice is based on Zolotarev functions and discussed in Section \ref{Subsec:DefZolo}.

\section{Structure preserving divide-and-conquer methods}\label{Sec:StrucDivConq}
The property of the matrix sign function to acquire invariant subspaces was used in the original paper \cite{Rob80} to solve algebraic Riccati equations. Later, it was used as a building block to devise parallelizable methods for eigenvalue computations of nonsymmetric matrices \cite{BaiD93,SunQ04,BenKS14}. In \cite{NakH13} a spectral divide-and-conquer algorithm for symmetric matrices is formulated, based on the relation between the matrix sign function and the polar decomposition. In this section, we generalize this approach to pseudosymmetric matrices. They are defined using signature matrices, which are diagonal matrices $\Sigma=\diag{\sigma_1,\dots,\sigma_n}$, where $\sigma_i\in\{1,-1\}$ for $i=1,\dots,n$.

\begin{definition}\label{Def:pseudosym}
 A matrix $A\in\mathbb{K}^{n\times n}$ is called \emph{pseudosymmetric (pseudo-Hermitian)} if there exists a signature matrix $\Sigma$, such that $A$ is self-adjoint with respect to the bilinear form (sesquilinear form) induced by $\Sigma$.
\end{definition}

Definition \ref{Def:pseudosym} is equivalent to $\Sigma A$ (or $A\Sigma$) being symmetric. Essentially, a pseudosymmetric matrix is symmetric up to sign changes of certain rows (or columns). This definition is slightly different than the one given, e.g., in \cite{MacMT05}, as we allow any signature matrix and not just $\Sigma_{p,q}=\begin{bmatrix}I_p&\\&-I_q\end{bmatrix}$. 

In Section \ref{sec:GenSpecDiv2} we outline the general idea of spectral division, which reduces a large eigenvalue problem to two smaller ones. Recursively applying this technique yields parallelizable methods for acquiring all eigenvalues and eigenvectors. Section \ref{Sec:SymSpecDiv} recounts how a symmetric structure can be preserved in this context. The same line of argument is applied to pseudosymmetric  matrices in Section \ref{Sec:PseuSymSpecDiv}. 

\subsection{General spectral divide-and-conquer}\label{sec:GenSpecDiv2}
It is a well-known concept to use invariant subspaces of a matrix to block-triangularize it with a similarity transformation. In the following we focus on real matrices, but everything extends to complex matrices. For real matrices we end up with $2\times 2$ matrix blocks on the diagonal for complex eigenvalues, whereas for complex matrices, this is unnecessary.

\begin{theorem}\label{Thm:GenSpecDiv}
  Let $A\in\mathbb{R}^{n\times n}$and $V_1\in\mathbb{R}^{n\times k}$  be a basis for an invariant subspace of $A$ and 
 $V=\begin{bmatrix}
     V_1&V_2
    \end{bmatrix}\in\mathbb{R}^{n\times n}$ have full rank. Then
    \begin{align*}
     V^{-1}AV = \begin{bmatrix}
                 A_{11} & A_{21}\\
                 0 & A_{22}
                \end{bmatrix}, \qquad A_{11}\in\mathbb{R}^{k\times k},\ A_{22}\in\mathbb{R}^{(n-k)\times (n-k)}.
    \end{align*}
\end{theorem}


Recursively applying the idea of Theorem \ref{Thm:GenSpecDiv} with shifts leads to a divide-and-conquer scheme, given in Algorithm \ref{Alg:GenSpecDiv}.

\begin{algorithm}[ht]
\caption{Unstructured spectral divide-and-conquer\label{Alg:GenSpecDiv}}
\begin{algorithmic}[1]
\Require $A\in\mathbb{R}^{n\times n}$
\Ensure $V, T$ such that $V^{-1}AV=T$ is block-upper triangular.
\State Stop if $A$ is of size $1\times 1$ or $2\times 2$ with a complex pair of eigenvalues.
\State Find shift $\sigma$ such that $A-\sigma I$ has eigenvalues with positive and negative real part and no eigenvalues with zero real part.
\State Compute $S=\sign{A-\sigma I}$ via an iteration.
\State \label{Step:Basis}Compute a basis $V_+$ of $\text{range}(S+I)$ and $V_-$ such that $V_0=\begin{bmatrix}
                                                                         V_+ & V_-
                                                                        \end{bmatrix}$ has full rank. Then
\begin{align*}
 V_0^{-1}A
 V_0 = 
 \begin{bmatrix}
  A_{11} & A_{12}\\
  0 & A_{22}
 \end{bmatrix}.
\end{align*}
 \State Repeat spectral divide-and-conquer for $A_{11}$, i.e.\ find $V_1$ such that $V_1^{-1}A_{11}V_1=T_{11}$ is block-upper triangular.
 \State Repeat spectral divide-and-conquer for $A_{22}$, i.e.\ find $V_2$ such that $V_2^{-1}A_{22}V_2=T_{22}$ is block-upper triangular.
 \State $V\leftarrow V\begin{bmatrix}
             V_1&0\\
             0&V_2
             \end{bmatrix}$, 
             $T\leftarrow\begin{bmatrix}
                 T_{11} & V_1^ {-1}A_{12}V_2\\
                 0&T_{22}
                \end{bmatrix}$.        
\end{algorithmic}
\end{algorithm}
This algorithm serves as a prototype for structure preserving methods developed in the next subsections. The key idea is to choose the subspace basis in Step \ref{Step:Basis} in a way that preserves the structure in the spectral division.
\subsection{Symmetric spectral divide-and-conquer}
\label{Sec:SymSpecDiv}
In this section we consider the symmetric eigenvalue problem, i.e.\ $A=A^{\tran}$. A structure-preserving method requires the spectral division $V^{-1}AV$ to be symmetric. This is exactly fulfilled by orthogonal matrices, i.e.\ for matrices fulfilling $V^{-1}=V^{\tran}$. A structure-preserving variant of Theorem \ref{Thm:GenSpecDiv} for symmetric matrices is given in the following.
\begin{theorem} 
  Let $A=A^{\tran}\in\mathbb{R}^{n\times n}$ and $V_1\in\mathbb{R}^{n\times k}$  be a basis of an invariant subspace of $A$ and 
 $V=\begin{bmatrix}
     V_1&V_2
    \end{bmatrix}$ be orthogonal. Then
    \begin{align*}
     V^{-1}AV=V^{\tran}AV = \begin{bmatrix}
                 A_{11} & 0\\
                 0 & A_{22}
                \end{bmatrix},\qquad A_{11}=A_{11}^{\tran}\in\mathbb{R}^{k\times k},\ A_{22}=A_{22}^{\tran}\in\mathbb{R}^{(n-k)\times(n- k)}.
    \end{align*}
\end{theorem}
The symmetric version of Algorithm \ref{Alg:GenSpecDiv} follows immediately as Algorithm  \ref{Alg:SymSpecDiv}. 
\begin{algorithm}[ht]
\caption{Symmetric spectral divide-and-conquer}\label{Alg:SymSpecDiv}
\begin{algorithmic}[1]
\Require $A=A^{\tran}\in\mathbb{R}^{n\times n}$
\Ensure Orthogonal $V$, diagonal $D$ such that $V^{\tran}AV=D$.
\State Stop if $A$ is of size $1\times 1$.
\State Find shift $\sigma$ such that $A-\sigma I$ has positive and negative eigenvalues and no zero eigenvalues.
\State Compute $S=\sign{A-\sigma I}$ via an iteration.
\State Compute a basis $V_+$ of $\text{range}(S+I)$ and $V_-$ such that $V_0=\begin{bmatrix}
                                                                         V_+ & V_-
                                                                        \end{bmatrix}$ is orthogonal. Then
\begin{align*}
 V_0^{\tran}A
 V_0 = 
 \begin{bmatrix}
  A_{11} & 0\\
  0 & A_{22}
 \end{bmatrix},\quad A_{11}=A_{11}^{\tran},\ A_{22}=A_{22}^{\tran}.
\end{align*}
 \State Repeat spectral divide-and-conquer for $A_{11}$, i.e.\ find $V_1$ such that $V_1^{\tran}A_{11}V_1=D_{11}$ is diagonal.
 \State Repeat spectral divide-and-conquer for $A_{22}$, i.e.\ find $V_2$ such that $V_2^{\tran}A_{22}V_2=D_{22}$ is diagonal.
 \State $V\leftarrow V_0\begin{bmatrix}
             V_1&0\\
             0&V_2
             \end{bmatrix}$, 
             $D\leftarrow\begin{bmatrix}
                 D_{11} & 0\\
                 0&D_{22}
                \end{bmatrix}$.        
\end{algorithmic}
\end{algorithm}

Due to the symmetry of $A$ and by restricting the subspace basis to be orthogonal, this can become a highly viable method. For symmetric $A$, $\sign{A}$ can be computed in a stable way via the QDWH iteration \cite{NakBG10,NakH12} or the Zolotarev iteration \cite{NakF16}. The basis extraction can be done by performing a rank-revealing QR decomposition or a subspace iteration \cite{NakH13}, if pivoting is considered too expensive.

\subsection{Pseudosymmetric spectral divide-and-conquer}
\label{Sec:PseuSymSpecDiv}
We now extend Section \ref{Sec:SymSpecDiv} to pseudosymmetric matrices.
The role of structure-preserving similarity transformations wmatas played by orthogonal matrices in Section \ref{Sec:SymSpecDiv}. For pseudosymmetric matrices this role is played by $(\Sigma,\hat{\Sigma})$-orthogonal matrices. 

\begin{lemma}
If $V\in\mathbb{R}^{m\times n}$ is a \emph{$(\Sigma,\hat{\Sigma})$-orthogonal} matrix and $A\in\mathbb{R}^{m\times m}$ is pseudosymmetric with respect to $\Sigma$, i.e.\ $\Sigma A = A^T\Sigma $. Then $\hat{A}=V^{\dagger}AV$ is pseudosymmetric with respect to $\hat{\Sigma}$, i.e.\ $\hat{\Sigma}\hat{A} = \hat{A}^T\hat{\Sigma}$.
\end{lemma}
\begin{proof}
With $V^{\dagger}= \hat{\Sigma} V^T \Sigma$, $\Sigma V = (V^{\dagger})^T\hat{\Sigma}$, $\Sigma^2=I_m$ and $\hat{\Sigma}^2=I_n$ we have
\begin{align*}
 \hat{\Sigma}(V^{\dagger}AV) = V^{\tran}{\Sigma}AV = V^{\tran}A^{\tran}{\Sigma}V = V^{\tran}A^{\tran}(V^\dagger)^{\tran}\hat{\Sigma} = (V^{\dagger}AV)^{\tran}\hat{\Sigma}.
\end{align*}
\end{proof}

What $({\Sigma},\hat{\Sigma})$-orthogonal matrices have in common with orthogonal matrices is that their \mbox{(pseudo-)} inverses can be easily computed via \eqref{Eq:pseudoinverse} in the form of
\begin{align*}
 V^{\dagger} = \hat{\Sigma}V^{\tran}{\Sigma}.
\end{align*}
For square matrices it holds $V^{-1}=V^{\dagger}$ and $V^{\dagger}AV$ constitutes a similarity transformation.

Methods for computing these matrices include the $HR$ decomposition \cite{BunB85} and methods described in \cite{BenNP22}. They prescribe $\Sigma$ and yield $\hat{\Sigma}$ and the $(\Sigma,\hat{\Sigma})$-orthogonal matrix $H$. We do not actually care about how $\hat{\Sigma}$ looks exactly, as long as it is a signature matrix. This way, pseudosymmetry as we defined it in Definition \ref{Def:pseudosym}, not being  bound to a specific $\Sigma$, is preserved. These kind of matrices, i.e.\ $(\Sigma,\hat{\Sigma})$-orthogonal matrices, where $\hat{\Sigma}$ does not matter, are sometimes called ``hyperexchange'' (e.g.\ in \cite{Seg10,Seg14}).

These observations can be used to formulate a pseudosymmetric variant of Algorithm \ref{Alg:GenSpecDiv}, given in Algorithm \ref{Alg:PseuDivConq}. In this algorithm, the property preserved in the spectral division is the pseudosymmetry. This means that ${\Sigma}$ does not stay fixed, but is permuted and truncated in each division step. 

\begin{algorithm}[t]
\caption{Pseudosymmetric spectral divide-and-conquer}
\label{Alg:PseuDivConq}
\begin{algorithmic}[1]
\Require Signature matrix ${\Sigma}$, pseudosymmetric $A$ with respect to $\Sigma$, i.e.\ $\Sigma A=(\Sigma A)^{\tran}$.
\Ensure Signature matrix $\hat{\Sigma}$, and  $({\Sigma},\hat{\Sigma})$-orthogonal $V$ such that $V^\dagger AV=D$ is block-diagonal with blocks no larger than $2\times 2$.
\State Stop if $A$ is of size $1\times 1$ or $2\times 2$ with a complex pair of eigenvalues.
\State Find shift $\sigma$ such that $A-\sigma I$ has eigenvalues with positive and negative real part and no eigenvalues with zero real part.
\State\label{Alg:PseuDivConq:Step:Sign}Compute $S=\sign{A-\sigma I}$ via an iteration.
\State\label{Alg:PseuDivConq:Step:Basis}Compute a basis $V_+$ of $\text{range}(S+I)$ and $V_-$ such that $V_0=\begin{bmatrix}
                                                                         V_+ & V_-
                                                                        \end{bmatrix}$ is $({\Sigma},\Sigma_{0})$-orthogonal with $\Sigma_{0}=\begin{bmatrix}
                                                                        \Sigma_{+}&\\
                                                                        &\Sigma_{-}\end{bmatrix}$. Then
\begin{align*}
 V_0^{\dagger}A
 V_0 = 
 \begin{bmatrix}
  A_{11} & 0\\
  0 & A_{22}
 \end{bmatrix},\quad \Sigma_{+}A_{11}=(\Sigma_{+}A_{11})^{\tran},\ \Sigma_{-}A_{22}=(\Sigma_{-}A_{22})^{\tran}.
\end{align*}
 \State Repeat Spectral divide-and-conquer for $A_{11}$ with ${\Sigma}:=\Sigma_+$, i.e.\ find $(\Sigma_+,\Sigma_1)$-orthogonal $V_1$ such that $V_1^\dagger A_{11}V_1=D_{11}$ is block-diagonal.

  \State Repeat Spectral divide-and-conquer for $A_{22}$ with ${\Sigma}:=\Sigma_{-}$, i.e.\ find $(\Sigma_-,\Sigma_2)$-orthogonal $V_2$ such that $V_2^\dagger A_{22}V_2=D_{22}$ is block-diagonal.
 \State $V\leftarrow V_0\begin{bmatrix}
             V_1&0\\
             0&V_2
             \end{bmatrix}$, $\hat{\Sigma} \leftarrow \begin{bmatrix}\Sigma_1&\\&\Sigma_2\end{bmatrix}$, 
             $D\leftarrow\begin{bmatrix}
                 D_{11} & 0\\
                 0&D_{22}
                \end{bmatrix}$.        
\end{algorithmic}
\end{algorithm}



\section{Computing $(\Sigma,\hat{\Sigma})$-orthogonal representations of subspaces}
\label{Sec:HypBasis}
Symmetric spectral divide-and-conquer methods rely on variants of the QR decomposition. The natural generalization in the indefinite context is the hyperbolic QR decomposition. 

\begin{proposition}[The hyperbolic QR decomposition \cite{BunB85}]\label{Thm:HypQR}
 Let $\Sigma\in\mathbb{R}^{m\times m}$ be a signature matrix, $A\in\mathbb{R}^{m\times n},\ m\geq n$. Suppose all the leading principal submatrices of $A^{\tran}\Sigma A$ are nonsingular. Then there exists a permutation $P$, a signature matrix $\hat{\Sigma}=P^{\tran}\Sigma P$, a $(\Sigma,\hat{\Sigma})$-orthogonal matrix $H\in\mathbb{R}^{m\times n}$ (i.e.\ $H^{\tran}\Sigma H= \hat{\Sigma}$), and an upper triangular matrix $R\in\mathbb{R}^{n\times n}$, such that
 \begin{align*}
  A=H\begin{bmatrix}
R\\
0
     \end{bmatrix}.
 \end{align*}
\end{proposition}

Similar to the orthogonal QR decomposition, it can be computed by applying transformations that introduce zeros below the diagonal, column by column. Details can e.g.\ be found in \cite{Wat07}. In \cite{SinS00}, the indefinite QR decomposition is presented, which improves stability by allowing $2\times 2$ blocks on the diagonal of $R$ and additional pivoting. This variant can also be computed via the (pivoted) $LDL^{\tran}$ decomposition of $A^{\tran}\Sigma A$; a link which was exploited in \cite{BenNP22} and \cite{BenP22}. There, the stability is improved by applying this method twice.

In the context of this work we aim to compute the indefinite QR decomposition of a pseudosymmetric projection matrix. We will see that in this special case, an indefinite QR decomposition can be computed via the $LDL^{\tran}$ decomposition without the need to form $A^{\tran}\Sigma A$. We do not make any statement about the stability of the proposed computations, as these considerations go beyond the scope of this paper, but make empirical observations in the numerical experiments presented in Section \ref{Sec:NumRes}.

We start with an observation regarding the symmetric divide-and-conquer method. Here, the matrix sign function computes a symmetric projection matrix, representing an orthogonal projection. 

\begin{lemma}\label{Lem:QRisChol} Let
 $P\in\mathbb{R}^{n\times n}$ be an orthogonal projection matrix, i.e.\ $P^2=P$ and   $P=P^{\tran}$, with rank $r$. Let $R^TR=P$, where $R\in\mathbb{R}^{r\times n}$, be a low-rank Cholesky factorization, where $R$ has full row rank. Then $R^T$ has orthogonal columns, i.e.\ $RR^{\tran}=I_r$, and $R^TR=P$ is a thin QR decomposition of P.
\end{lemma}
\begin{proof}
 Because P is positive semi-definite, the low-rank Cholesky factorization exists. From $P=P^2$ follows $R^{\tran}R=R^{\tran}RR^{\tran}R$ and therefore $RR^{\tran}=I_r$.
\end{proof}

Lemma \ref{Lem:QRisChol} states that for projection matrices attained via the matrix sign function, the low-rank Cholesky and the thin QR decomposition are equivalent.

Let $P_+=\frac{1}{2}(I_n+\sign{A})$ be the projection on the subspace of $A$ associated with positive eigenvalues. The advantage of computing the (full) QR decomposition $\begin{bmatrix}Q_+&Q_-\end{bmatrix}\begin{bmatrix}R\\0\end{bmatrix}$ is that we immediately get a basis $Q_-$ for the complementing subspace, associated with negative eigenvalues. The Cholesky factorization applied in the sense of Lemma \ref{Lem:QRisChol} can only yield a thin QR decomposition. However, the same procedure can be applied to $P_-=\frac{1}{2}(I_n-\sign{A}$. The two thin QR decompositions can be combined to form a full one. Indeed, let $Q_+$ and $Q_-$ be acquired from $P_+$ and $P_-$ via Lemma \ref{Lem:QRisChol}. The identities $Q_+^{\tran}Q_+=I$ and $Q_-^{\tran}Q_- =I$ follow immediately from the orthogonality proven in the lemma. From $P_+=Q_+Q_+^{\tran}$ follows $Q_+{^{\tran}}=Q_+^{\tran}P_+$ and from $P_-=Q_-Q_-^{\tran}$ follows $Q_-=P_-Q_-$. From the definition of the projectors in Lemma \ref{Lem:SignProj} we have $P_+P_-=0$ and therefore  $Q_+^{\tran}Q_- = Q_+P_+^{\tran}P_-Q_- = 0 $.

Algorithm \ref{Alg:OrthSubspaceChol} shows how Lemma \ref{Lem:QRisChol} can be used to compute an orthogonal representation of an invariant subspace of a symmetric matrix. In Step \ref{Alg:OrthSubspaceChol:Rank} we use the trace of a projection matrix to determine its rank. For badly conditioned matrices, pivoting could be included in the computation of the Cholesky factorization. Lemma \ref{Lem:QRisChol} does not need to assume the triangular shape of $R$ to show that its rows are orthogonal. 

\begin{algorithm}
\caption{Compute orthogonal invariant subspace representations of a symmetric matrix via Cholesky.}\label{Alg:OrthSubspaceChol}
 \begin{algorithmic}[1]
  \Require $A=A^{\tran}\in\mathbb{R}^{n\times n}$ nonsingular.
  \Ensure An orthogonal basis $Q=\begin{bmatrix}
                                  Q_+ & Q_-
                                 \end{bmatrix}$, where $Q_+$ is a basis of the invariant subspace of $A$ associated with positive eigenvalues, $Q_-$ is a basis of the invariant subspace of $A$ associated with negative eigenvalues.  
 \State $S\leftarrow \sign{A}$.
 \State $P_+ \leftarrow \frac{1}{2} (I_n+S)$.
 \State Compute $\rank{P_+}=:r_+\leftarrow\trace{P_+}$.\label{Alg:OrthSubspaceChol:Rank}
 \State $Q_{1,+}\leftarrow \texttt{chol}(P_+(1:r_+,1:r_+)$.
 \State $Q_+\leftarrow\begin{bmatrix}
            Q_{1,+}\\
            P_+(r_++1:n,1:r_+)Q_{1,+}^{-\tran}
           \end{bmatrix}$.\label{Alg:OrthSubspaceChol:extendChol1}
  \State $P_- \leftarrow \frac{1}{2} (I_n-S)$.
 \State Compute $\rank{P_-}=:r_-\leftarrow n-r_+$.
 \State $Q_{1,-}\leftarrow \texttt{chol}(P_-(1:r_-,1:r_-)$.
 \State $Q_-\leftarrow\begin{bmatrix}
            Q_{1,-}\\
            P_+(r_-+1:n,1:r_-)Q_{1,-}^{-\tran}
           \end{bmatrix}$.\label{Alg:OrthSubspaceChol:extendChol2}
 \end{algorithmic}
\end{algorithm}

In the symmetric context, computing the QR decomposition like this does not have an obvious benefit over computing a QR decomposition the standard way. However, it can be generalized to the indefinite case. Here, an $LDL^{\tran}$ decomposition can be used instead of a hyperbolic QR decomposition, which is much more widely used. Established algorithms and highly-optimized implementations are available and ready to use, e.g.\ in MATLAB as \texttt{ldl} command. Details are given in the following theorem.

\begin{theorem}\label{Thm:HRisLDL} 
$\Sigma$ is a given signature matrix, $P\in\mathbb{R}^{n\times n}$ is a projection matrix and pseudosymmetric with respect to $\Sigma$, i.e.\ $P^2=P$ and   $\Sigma P\Sigma=P^{\tran}$, with rank $r$. Let $R^T\hat{\Sigma}R=\Sigma P$, where $R\in\mathbb{R}^{r\times n}$, be a scaled low-rank $LDL^{\tran}$ factorization, where $R$ has full row rank and $\hat{\Sigma}\in\mathbb{R}^{r\times r}$ is another signature matrix. Then $R^T$ is $(\Sigma,\hat{\Sigma})$-orthogonal,, i.e.\ $R{\Sigma}R^{\tran}=\hat{\Sigma}$, and $HR=P$ with $H=\Sigma R^{\tran} \hat{\Sigma}$ is a decomposition of P, where $H$ is $(\Sigma,\hat{\Sigma})$-orthogonal.
\end{theorem}
\begin{proof}
With $P=\Sigma R^{\tran} \hat{\Sigma} R$ and $P=P^2$ it follows $\Sigma R^{\tran} \hat{\Sigma} R=\Sigma R^{\tran} \hat{\Sigma} R\Sigma R^{\tran} \hat{\Sigma} R$ and therefore 
\begin{align}\label{Eq:RSigmaOrth}
\hat{\Sigma} = \hat{\Sigma} R\Sigma R^{\tran} \hat{\Sigma}\quad
\Leftrightarrow\quad \hat{\Sigma}=R\Sigma R^{\tran}.
\end{align}
We used $\hat{\Sigma}^2=I_r$. \eqref{Eq:RSigmaOrth} is equivalent to $H:=R^{\dagger}=\Sigma R^{\tran} \hat{\Sigma}$ being 
$(\Sigma,\hat{\Sigma})$-orthogonal: $H^{\tran}\Sigma H = \hat{\Sigma}$. We therefore have a decomposition $P=\Sigma R^{\tran} \hat{\Sigma} R=HR$.
\end{proof}
If $R$ in Theorem \ref{Thm:HRisLDL} is computed with the Bunch-Kaufman algorithm \cite{BunK77} (e.g.\ MATLAB \texttt{ldl}), it can be a permuted block-triangular matrix and stability can be improved. Then $P=HR$ is not a hyperbolic QR decomposition in the strict sense given in Theorem \ref{Thm:HypQR}. This is not important here, as we are only interested in the subspace given by $H$. 
The indefinite variant of Algorithm \ref{Alg:OrthSubspaceChol} is given in Algorithm \ref{Alg:HypSubspaceLDL}.
\begin{algorithm}[t]
\caption{Compute hyperbolic invariant subspace representations of a pseudosymmetric projection matrix via $LDL^{\tran}$} \label{Alg:HypSubspaceLDL}
 \begin{algorithmic}[1]
  \Require Signature matrix $\Sigma$, $A=\Sigma A^{\tran}\Sigma\in\mathbb{R}^{n}$ nonsingular.
  \Ensure A signature matrix $\hat{\Sigma}$, which is a permuted variant of $\Sigma$, a $(\Sigma,\hat{\Sigma})$-orthogonal basis $Q=\begin{bmatrix}
                                  Q_+ & Q_-
                                 \end{bmatrix}$, i.e.\ $Q^{\tran}\Sigma Q= \hat{\Sigma}$, where $Q_+$ is a basis of the invariant subspace of $A$ associated with positive eigenvalues, $Q_-$ is a basis of the invariant subspace of $A$ associated with negative eigenvalues.  
 \State $S\leftarrow \sign{A}$.
 \State $P_+ \leftarrow \frac{1}{2} (I_n+S)$.
 \State Compute $\rank{P_+}=:r_+\leftarrow\trace{P_+}$.\label{Alg:HypSubspaceLDL:Rank}
 \State $[L_+,D_+]\leftarrow \texttt{ldl}(\Sigma P_+)$.
 \State Diagonalize $D_+$ if it has blocks on the diagonal: $[V_+,D_+]\leftarrow\texttt{eig}(D_+)$, such that $D_+(1:r_+,1:r_+)$ contains the nonzero diagonal values of $D_+$.
 \State $R_+\leftarrow (L_+V_+(:,1:r_+)D_+(1:r_+,1:r_+)^{\frac{1}{2}})^{\tran}$, \quad $\hat{\Sigma}_+\leftarrow\sign{D_+(1:r_+,1:r_+)}$.\label{Alg:HypAlg:HypSubspaceLDL:Trunc1}

 \State $P_- \leftarrow \frac{1}{2} (I_n-S)$.
 \State Compute $\rank{P_-}=:r_-\leftarrow n-r_+$.
 \State $[L_-,D_-]\leftarrow \texttt{ldl}(\Sigma P_-)$.
 \State Diagonalize $D_-$ if it has blocks on the diagonal: $[V_-,D_-]\leftarrow\texttt{eig}(D_-)$, such that $D_-(1:r_-,1:r_-)$ contains the nonzero diagonal values of $D_-$.
  \State $R_-\leftarrow (L_-V_-(:,1:r_-)D_-(1:r_-,1:r_-)^{\frac{1}{2}})^{\tran}$,\quad  $\hat{\Sigma}_-\leftarrow\sign{D_-(1:r_-,1:r_-)}$.\label{Alg:HypAlg:HypSubspaceLDL:Trunc2}
  \State $\hat{\Sigma}\leftarrow\diag{\Sigma_+,\Sigma_-}$.
  \State $Q_+ \leftarrow \Sigma R_+^{\tran} \hat{\Sigma}$,\quad $Q_- \leftarrow \Sigma R_-^{\tran} \hat{\Sigma}$.
 \end{algorithmic}
\end{algorithm}

In contrast to the MATLAB function \texttt{chol}, \texttt{ldl} is not affected by singular matrices, such as the given projectors.  This is why steps \ref{Alg:OrthSubspaceChol:extendChol1} and \ref{Alg:OrthSubspaceChol:extendChol2} in Algorithm \ref{Alg:OrthSubspaceChol} do not have a correspondence in Algorithm \ref{Alg:HypSubspaceLDL}. The Cholesky-based algorithm (Algorithm \ref{Alg:OrthSubspaceChol}) computes the Cholesky factorization of the upper left block and expands it in order to get a low-rank version. The $LDL^{\tran}$-based algorithm (Algorithm \ref{Alg:HypSubspaceLDL}) on the other hand computes an $LDL^{\tran}$ decomposition of the whole matrix, which we then truncate in Steps \ref{Alg:HypAlg:HypSubspaceLDL:Trunc1} and \ref{Alg:HypAlg:HypSubspaceLDL:Trunc2}.

\section{Definite pseudosymmetric matrices}\label{Sec:DefPseudosym}
In this section we consider pseudosymmetric matrices with an additional property. We call a pseudosymmetric matrix $A$ with respect to a signature matrix $\Sigma$ \emph{definite} if
$\Sigma A$ is positive definite. 

The Bethe-Salpeter equation (BSE) approach is a state-of-the art method for computing optical properties of materials and molecules, derived from many-body perturbation theory. After appropriate discretization, eigenvalues and eigenvectors of a complex structured matrix
\begin{align}\label{Eq:BSE1}
 H_{\text{BSE}} = \begin{bmatrix}
            A_{\text{BSE}} & B_{\text{BSE}}\\
            -B_{\text{BSE}}^{\herm} & -A_{\text{BSE}}^{\tran}
           \end{bmatrix},\qquad A_{\text{BSE}}=A_{\text{BSE}}^{\herm},\quad B_{\text{BSE}}=B_{\text{BSE}}^{\tran}
\end{align}
are sought \cite{OniRR02}. A similar eigenvalue problem arises when molecules are considered within time-dependent density functional theory in the linear response regime. Here, the Casida equation can be recast into an eigenvalue problem of the real matrix
\begin{align}\label{Eq:Casida}
 H_{\text{Cas}} = \begin{bmatrix}
            A_{\text{Cas}} & B_{\text{Cas}}\\
            -B_{\text{Cas}}& -A_{\text{Cas}}
           \end{bmatrix},\qquad A_{\text{Cas}}=A_{\text{Cas}}^{\tran},\quad B_{\text{Cas}}=B_{\text{Cas}}^{\tran}.
\end{align}
Considering a Bethe-Salpeter approach within Hartree-Fock theory for molecules leads to a matrix with the same structure \cite{BenKK16}.

For crystalline solids, the periodic structure can be exploited and with a proper choice of basis functions the resulting matrix has the form \cite{SanMK15}
\begin{align}\label{Eq:BSE2}
 H_{\text{BSE},2} = \begin{bmatrix}
            A_{\text{BSE},2} & B_{\text{BSE},2}\\
            -B_{\text{BSE},2} & -A_{\text{BSE},2}
           \end{bmatrix}
           ,\qquad A_{\text{BSE},2}=A_{\text{BSE},2}^{\herm},\quad B_{\text{BSE},2}=B_{\text{BSE},2}^{\herm}.
\end{align}
The setup \eqref{Eq:BSE2} is essentially a complex version of \eqref{Eq:Casida}. A more detailed analysis of the special structure in \eqref{Eq:BSE1} and \eqref{Eq:BSE2} is given in \cite{BenP22}.

All of these matrices are obviously pseudosymmetric with respect to $\Sigma=\diag{I,-I}$. Furthermore, they are typically definite, i.e.\ $\Sigma H$ is positive definite for any $H$ defined in \eqref{Eq:BSE1}, \eqref{Eq:Casida} or \eqref{Eq:BSE2}.

\subsection{Decoupling the indefinite eigenvalue problem into two symmetric definite problems}


In the following, we explain how the spectral divide-and-conquer algorithm described in Section \ref{Sec:PseuSymSpecDiv} simplifies greatly for definite pseudosymmetric matrices. Essentially, the problem can be reduced to two Hermitian positive definite eigenvalue problems after just one spectral division step.


As a first result, we present the following theorem, clarifying the spectral structure of definite pseudosymmetric matrices. It is an extension of Theorem 5 in \cite{BenP22}, additionally clarifying the structure of the eigenvectors, and a more general variant of Theorem 3 in \cite{ShaDYetal16}. Our version is independent of the additional structure of Bethe-Salpeter matrices given in \eqref{Eq:BSE1}. It can be proven in a similar fashion relying on the simultaneous diagonalization of $\Sigma A$ and $\Sigma$. 
%

\begin{theorem}\label{Thm:EVREalSigmaOrth}
Let $A\in\mathbb{K}^{n\times n}$ be a definite pseudosymmetric matrix with respect to $\Sigma$, where $\Sigma$ has $p$ positive and $n-p$ negative diagonal entries.  Then A has only real, nonzero eigenvalues, of which $p$ are positive and $n-p$ are negative. There is an eigenvalue decomposition $AV = V \Lambda$, $\Lambda=\diag{\lambda_1,\dots,\lambda_n}$, where $\lambda_1,\dots, \lambda_p >0$, $\lambda_{p+1},\dots, \lambda_n<0$, such that\begin{align}\label{Eq:VSigmaOrth}
V^*\Sigma V=\begin{bmatrix}I_p&\\&-I_{n-p}\end{bmatrix}.                                                                                                                                                                                                                                                                                                                                                                                                                                                                                                                                                                                                                                                                                                                                       \end{align}
\end{theorem}

\begin{proof}
As $\Sigma A$ is positive definite, and $\Sigma$ is symmetric, they can be diagonalized simultaneously (see \cite{GolV13}, Corollary 8.7.2), i.e. there is a nonsingular $X\in\mathbb{C}^{n\times n}$ s.t.
$X^{\herm}\Sigma A X = I_n$, and $X^{\herm} \Sigma X = \Lambda^{-1} \in\mathbb{R}^{n\times n}$,
 where $\Lambda^{-1}=\diag{\lambda_1^{-1},\dots,\lambda_n^{-1}}$ gives the eigenvalues of the matrix pencil $\Sigma x - \lambda \Sigma A$.  It follows from  Sylvester's law of inertia that $\Lambda^{-1}$ has $p$ positive and $n-p$ negative values. We have
  $X^{-1}AX= \Lambda$,
 i.e. $A$ is diagonalizable and $\Lambda^{-1}$ contains the eigenvalues of $A$. 
The columns of $X$ can be arranged, such that the positive eigenvalues are given in the upper left part of $\Lambda$ and the negative ones are given in the lower right part. $X$ can be scaled in form of $V:=X{|\Lambda|}^{-\frac{1}{2}}$, where $|\cdot|$ denotes the entry-wise absolute value, such that \eqref{Eq:VSigmaOrth} holds.
\end{proof}

For pseudosymmetric matrices that are definite, the structure-preserving spectral divide-and-conquer algorithm (Algorithm \ref{Alg:PseuDivConq}) shows a special behaviour that can be exploited algorithmically. Generally, after one step of spectral division, we get two smaller matrices that are pseudosymmetric with respect to two submatrices of the original signature matrix $\Sigma$, denoted $\Sigma_+$ and $\Sigma_-$ in Algorithm \ref{Alg:PseuDivConq}. The $p$ positive and the $n-p$ negative values on the diagonal of $\Sigma$ split up in an unpredictable way. For definite matrices they split up neatly: The positive values gather in $\Sigma_+=I_p$ and the negative values gather in $\Sigma_- = -I_{n-p}$. After spectral division, the upper left block  $A_{11}$ is definite pseudosymmetric with respect to $I_p$, i.e.\ symmetric positive definite. The lower right block $A_{22}$ is definite pseudosymmetric with respect to $-I_{n-p}$, i.e.\ symmetric negative definite. This behavior is explained in the following theorem.

\begin{theorem}\label{Thm:DefDivConq}
Let $A\in\mathbb{K}^{n\times n}$ be a definite pseudosymmetric matrix with respect to $\Sigma$ with $p$ positive and $n-p$ negative diagonal values. Let $H$ be a basis of the invariant subspace of $A$ associated with the $p$ positive (respectively $n-p$ negative) eigenvalues, such that $H^*\Sigma H=\hat{\Sigma}$, where $\hat{\Sigma}$ is another signature matrix.  Then $H^{\dagger} A H$ is Hermitian positive (respectively negative) definite and $\hat{\Sigma}=I_p$ (respectively $\hat{\Sigma}=-I_{n-p}$).    
\end{theorem}
\begin{proof}
  Let $AV=V\Lambda$ be the eigenvalue decomposition given in Theorem \ref{Thm:EVREalSigmaOrth}. Let $V_p=\begin{bmatrix}v_1&\dots v_p\end{bmatrix}$ denote the first $p$ columns of $V$, associated with the positive eigenvalues $\Lambda_+ = \diag{\lambda_1,\dots,\lambda_p}$. Then $AV_+=V_+\Lambda_+$ and 
\begin{align}\label{Eq:VpSigmaVp}
V_+^*\Sigma V_+=I_p.
\end{align}

As $H$ spans the same subspace as $V_+$, there must be $X\in\mathbb{K}^{p\times p}$ such that $H=V_+X$. Then $H^*\Sigma H = X^*V_+^*\Sigma V_+ X=X^*X$ is positive definite.  The only signature matrix with this property is the identity, showing $\hat{\Sigma}=I_p$.

Then it holds $H^\dagger = H^*\Sigma$ and therefore
\begin{align*}
 H^\dagger A H = H^*\Sigma A H
\end{align*}
is Hermitian positive definite, as  $\Sigma A$ is Hermitian positive definite.
Concerning the negative eigenvalues it can be shown that $\hat{\Sigma}=-I_{n-p}$ and therefore 
\begin{align*}
 H^\dagger A H = -H^*\Sigma A H
\end{align*}
is Hermitian negative definite.
\end{proof}

Theorem \ref{Thm:DefDivConq} greatly simplifies the divide-and-conquer method for definite pseudosymmetric matrices (Algorithm \ref{Alg:PseuDivConq}). We only need one spectral division step and can then fall back on existing algorithms for symmetric positive definite matrices. They can be of the divide-and-conquer variety, e.g.\ developed in \cite{NakH13}, but do not have to be. In a high-performance setting, parallelized algorithms implemented in libraries such as ELPA \cite{MarBJetal14} can be used. 

\subsection{Computing $(\Sigma,\hat{\Sigma})$-orthogonal representations}
\label{Subsec:DefPseudoorthRep}

The computation of pseudoorthogonal subspace representations described in Section \ref{Sec:HypBasis} also simplifies. In Step \ref{Alg:HypAlg:HypSubspaceLDL:Trunc1} of Algorithm \ref{Alg:HypSubspaceLDL}, the smaller signature matrix $\Sigma_+$ related to the subspace associated with positive eigenvalues is computed by taking the signs of the diagonal matrix $D$ of the previously computed $LDL^{\tran}$ decomposition. Because of Theorem \ref{Thm:DefDivConq} we know that $\Sigma_+=I_p$. The $LDL^{\tran}$ decomposition was taken of $\Sigma P_+$, which hence must be positive semidefinite. Therefore, the $LDL^{\tran}$ decomposition can be substituted by a low-rank Cholesky factorization, similar to the symmetric case described in Algorithm \ref{Alg:OrthSubspaceChol}.  The computation of the rank (Step \ref{Alg:HypSubspaceLDL:Rank} in Algorithm \ref{Alg:HypSubspaceLDL}) is omitted because we know that $A$ has as many positive eigenvalues as $\Sigma$ has positive diagonal values according to Theorem \ref{Thm:EVREalSigmaOrth}.

\begin{algorithm}
\caption{Compute $(\Sigma,\hat{\Sigma})$-orthogonal invariant subspace representations of a definite pseudosymmetric projection matrix via Cholesky} \label{Alg:HypSubspaceChol}
 \begin{algorithmic}[1]
  \Require Signature matrix $\Sigma$ with $r_+$ positive and $r_-$ negative diagonal values, $A\in\mathbb{R}^{n}$, such that $\Sigma A$ is symmetric positive definite.
  \Ensure A$(\Sigma,\hat{\Sigma})$-orthogonal basis $Q=\begin{bmatrix}
                                  Q_+ & Q_-
                                 \end{bmatrix}$, where $\hat{\Sigma}=\diag{I_{r_+},-I_{r_-}}$, i.e.\ $Q^{\tran}\Sigma Q= \hat{\Sigma}$, where $Q_+$ is a basis of the invariant subspace of $A$ associated with positive eigenvalues, $Q_-$ is a basis of the invariant subspace of $A$ associated with negative eigenvalues.  
 \State $S\leftarrow \sign{A}$.
 \State $P_+ \leftarrow \frac{1}{2} (I_n+S)$.
 \State $Q_{1,+}\leftarrow \texttt{chol}(\Sigma(1:r_+,1:r_+) P_+(1:r_+,1:r_+))$.
  \State $Q_+\leftarrow\begin{bmatrix}
            \Sigma(1:r_+,1:r_+)Q_{1,+}^{\herm}\\
            P_+(r_++1:n,1:r_+)Q_{1,+}^{-1}
           \end{bmatrix}$.         
  \State $P_- \leftarrow \frac{1}{2} (I_n-S)$.
 \State $Q_{1,-}\leftarrow \texttt{chol}(-\Sigma(1:r_-,1:r_-) P_-(1:r_-,1:r_-))$.
  \State $Q_-\leftarrow\begin{bmatrix}
            \Sigma(1:r_-,1:r_-)Q_{1,-}^{\herm}\\
            P_-(r_-+1:n,1:r_-)Q_{1,-}^{-1}
           \end{bmatrix}$.
 \end{algorithmic}
\end{algorithm}

Numerical experiments (in particular examples from electronic structure theory, presented in Section \ref{Subsec:Appl}) show that Algorithm \ref{Alg:HypSubspaceChol} can break down due to numerical errors in floating point arithmetic. This happens when numerical errors lead to $\Sigma P_+$ having negative eigenvalues or $\Sigma P_-$ having positive eigenvalues, such that the Cholesky decomposition breaks down. In order to avoid this case, we implement a more robust variant based on a truncated $LDL^{\tran}$ decompositions, which includes pivoting, given in Algorithm \ref{Alg:HypSubspaceDefLDL}.

\begin{algorithm}[t]
\caption{Robust computation of $(\Sigma,\hat{\Sigma})$-orthogonal invariant subspace representations of a definite pseudosymmetric projection matrix via $LDL^{\tran}$} \label{Alg:HypSubspaceDefLDL}
 \begin{algorithmic}[1]
  \Require Signature matrix $\Sigma$ with $r_+$ positive and $r_-$ negative diagonal values, $A\in\mathbb{R}^{n}$, such that $\Sigma A$ is symmetric positive definite.
    \Ensure A$(\Sigma,\hat{\Sigma})$-orthogonal basis $Q=\begin{bmatrix}
                                  Q_+ & Q_-
                                 \end{bmatrix}$, where $\hat{\Sigma}=\diag{I_{r_+},-I_{r_-}}$, i.e.\ $Q^{\tran}\Sigma Q= \hat{\Sigma}$, where $Q_+$ is a basis of the invariant subspace of $A$ associated with positive eigenvalues, $Q_-$ is a basis of the invariant subspace of $A$ associated with negative eigenvalues.  
 \State $S\leftarrow \sign{A}$.
 \State $P_+ \leftarrow \frac{1}{2} (I_n+S)$.
\State $[L_+,D_+]\leftarrow \texttt{ldl}(\Sigma P_+)$.
 \State Diagonalize $D_+$ if it has blocks on the diagonal: $[V_+,D_+]\leftarrow\texttt{eig}(D_+)$, such that the diagonal entries of $D_+$ are given in descending order. 
 \State $Q_+\leftarrow \Sigma L_+ V_+(:,1:r_+)D_+^{\frac{1}{2}}(1:r_+,1:r_+)$.
  \State $P_- \leftarrow \frac{1}{2} (I_n-S)$.
\State $[L_-,D_-]\leftarrow \texttt{ldl}(-\Sigma P_-)$.
 \State Diagonalize $D_-$ if it has blocks on the diagonal: $[V_-,D_-]\leftarrow\texttt{eig}(D_-)$, such that the diagonal entries of $D_-$ are given in descending order. 
 \State $Q_-\leftarrow \Sigma L_- V_-(:,1:r_-)D_-^{\frac{1}{2}}(1:r_-,1:r_-)$.
 \end{algorithmic}
\end{algorithm}

\subsection{Using Zolotarev functions to accelerate the matrix sign iteration}
\label{Subsec:DefZolo}
It was observed in \cite{BenNP22} that the matrix sign function of a self-adjoint matrix $A$ is given as the first factor of its generalized polar decomposition, offering a new perspective for its computation.  A matrix $A\in\mathbb{K}^{n\times n}$ (under certain assumptions, see \cite{HigMMetal05}) admits a generalized polar decomposition with respect to a scalar product induced by a nonsingular matrix $M$
\begin{align*}
A=WS,
\end{align*}
where $W$ is a partial $M$-isometry and $S$ is $M$-self-adjoint with no eigenvalues on the negative real axis. Canonical generalized polar decompositions can be defined for rectangular matrices \cite{HigMT10}. We only consider the square case relevant to the application discussed in this work. 

Iterations of a certain form that compute the generalized polar decomposition $A=WS$ work as a scalar iteration on the eigenvalues of $S$, pushing them closer to $1$ in the course of the iteration. The following lemma clarifies this idea and is a slightly altered variant of Theorem 5.2 in \cite{BenNP22}. 
\begin{lemma}\label{Lem:IterOnSelfAdj}
 Let $g$ be a scalar function of the form
 \begin{align} \label{Eq:IterOnSelfAdj:g}
  g(x) = x h(x^2),
 \end{align}
where $h$ is an arbitrary scalar function. Let $A\in\mathbb{R}^{n\times n}$ be a matrix with a generalized polar decomposition $A=WS$ for a given scalar product induced by $M\in\mathbb{R}^{n\times n}$. Let
\begin{align}\label{Eq:IterOnSelfAdj:G}
 G(X) := Xh(X^{\star_M}X)
\end{align}
be a matrix function. Then it holds
\begin{align*}
 G(A) = WG(S) = Wg(S).
\end{align*}
\end{lemma}
\begin{proof}
Observe
 \begin{align*}
  G(A)=G(WS)=WSh(S^{\star_M}W^{\star_M}WS) = WSh(S^{\star_M}S) = WG(S)=Wg(S).
 \end{align*}
 We used $W^{\star_M}WS=S$, which holds according to Lemma 3.7. in \cite{HigMT10}. The last equality holds because $S$ is self-adjoint and $S^{\star_M}S = S^2$.

\end{proof}

Given an iteration of the form
\begin{align}\label{Eq:GenIter}
 X_{k+1} = G(X_k), 
\end{align}
with $G$ from \eqref{Eq:IterOnSelfAdj:G}, Lemma \ref{Lem:IterOnSelfAdj} states that it acts as the function $g$ from \eqref{Eq:IterOnSelfAdj:g} on the eigenvalues of the self-adjoint factor $S$. With the Jordan decomposition $S=ZJZ^{-1}$, $J=\diag{J_k}$ we see
\begin{align}\label{Eq:WorkOnJblocks}
 X_{k+1} = WZ\diag{g(J_k)}Z^{-1}.
\end{align}

We have already seen in Theorem \ref{Thm:DefDivConq} that definite pseudosymmetric matrices have a special spectral structure. The following lemma shows that as a consequence, the eigenvalues of the self-adjoint factor $S$, on which iterations of the form \eqref{Eq:GenIter} act, are real.


\begin{lemma}\label{Lem:RationalIter}
 Let $A\in\mathbb{K}^{n\times n}$ be a definite pseudosymmetric matrix with respect to $\Sigma$. Then the generalized polar decomposition of $A$ with respect to $\Sigma$,
 \begin{align*}
  A=WS,
 \end{align*}
 exists. The eigenvalues of $S$ are positive real and the absolute values of the eigenvalues of $A$.
\end{lemma}
\begin{proof}
 For pseudosymmetric matrices it holds $A^{\star_\Sigma}A=\Sigma A^* \Sigma A = A^2$. As $A$ has only real nonzero eigenvalues (following from Theorem \ref{Thm:EVREalSigmaOrth}), $A^2$ has only real positive eigenvalues. Hence the generalized polar decomposition exists. The self-adjoint factor of the polar decomposition is defined as $S=(A^{\star_\Sigma}A)^{\frac{1}{2}}$ and has only real eigenvalues, as the square roots of real positive values are real. They are the absolute values of the eigenvalues of A.
\end{proof}
Let $A$ be scaled, such that its eigenvalues lie in $[-1,1]$, and let $0<\ell<|\lambda|$ for all $\lambda\in\Lambda(A)$. Then the eigenvalues of $S$ lie in $(\ell,1]$. A rational function $g(x)=xh(x^2)$ which maps them close to $1$, i.e.\ approximates the scalar sign function on the interval $(\ell,1]$, can be used in an iteration \eqref{Eq:GenIter}. We see from \eqref{Eq:WorkOnJblocks} that the result will be an approximation to the polar factor $W$, which in our setting coincides with the matrix sign function. Luckily, explicit formulas for rational best-approximations of the sign function with form \eqref{Eq:IterOnSelfAdj:g} were found by Zolotarev in 1877 \cite{Zol77}. In \cite{NakF16}, Zolotarev functions are used to devise an iteration which computes the polar decomposition in just two steps. The algorithmic cost of the steps is increased compared to other iterative techniques, but the additional computations can be performed completely in parallel.  We extend this approach for computing the polar decomposition of definite pseudosymmetric matrices.

We call the unique rational function of degree $(2r+1,2r)$ solving
\begin{align*}
 \min_{R\in\mathcal{R}_{2r+1,2r}} 
 \max_{x\in[-1,-\ell]\cup[\ell,1]} 
 |\sign{x}-R(x)|
\end{align*}
for a given $0<\ell<1$ and an integer $r$, the \emph{type $(2r+1,2r)$ Zolotarev function}. It is given explicitly in the form of
\begin{align}\label{Eq:Zolo}
 Z_{2r+1}(x;\ell):=Cx\prod_{j=1}^r
 \frac{x^2+c_{2j}}{x^2+c_{2j-1}}.
\end{align}
The coefficients $c_1,\dots, c_{2r}$ are determined via the Jacobi elliptic functions $\text{sn}(u;\ell)$ and $\text{cn}(u;\ell)$ as
\begin{align}\label{Eq:Compcis}
 c_i = \ell^2 \frac{\text{sn}^2(\frac{\iu K'}{2r+1};\ell')}{\text{cn}^2(\frac{\iu K'}{2r+1};\ell')},\qquad i=1,\dots 2r,
\end{align}
where $\ell'=\sqrt{1-\ell^2}$ and $K'=\int_0^{\pi/2}(1-(\ell')^2\sin^2(\theta))^{-1/2}d\theta$ are familiar quantities in the context of Jacobi elliptic functions (see e.g.\ \cite[Chapter 17]{AbrS72}, , \cite[Chapter 5]{Akh90}). Details on the stable computation of the coefficients can be found in \cite{NakF16}. The constant $C>0$ is uniquely determined, which will later be substituted by a normalization constant $\hat{C}$. We use implementations provided as MATLAB functions in \cite{NakF16} for their computation.

Zolotarev also showed (see \cite[Chapter 9]{Akh92}, \cite[Chapter 4]{PetP87}) that $Z_{2r+1}(x;\ell)$ solves 
\begin{align*}
 \max_{P,Q\in\mathcal{P}_r} \min_{\ell\leq x \leq 1} x{ \frac{P(x^2)}{Q(x^2)}}.
\end{align*}
For $r=1$, this optimization problem was solved in \cite{NakBG10}, leading to the dynamically weighted Halley (DWH) iteration. This iteration was used in \cite{BenNP22} to compute the generalized polar decomposition. An iteration based on Zolotarev functions therefore generalizes the DWH approach in terms of higher-degree rational functions. 

A key oberservation in \cite{NakF16} is that the composition of Zolotarev functions is again a Zolotarev function. More precisely, it holds
\begin{align*} 
 \hat{Z}_{2r+1}(\hat{Z}_{2r+1}(x;\ell);\ell_1) = \hat{Z}_{(2r+1)^2}(x;\ell),
\end{align*}
where 
\begin{align}\label{Eq:hatZ}
\hat{Z}_{2r+1}(x;\ell)= \frac{Z_{2r+1}(x;\ell)}{Z_{2r+1}(1;\ell)} = \hat{C}x\prod_{j=1}^r \frac{x^2+c_{2j}}{x^2+c_{2j-1}},\text{ with }\hat{C}=\prod_{j=1}^r\frac{1+c_{2j-1}}{1+c_{2j}},
\end{align}
is a scaled Zolotarev function and $\ell_1=\hat{Z}_{2r+1}(\ell;\ell)$. It can be verified that with $r:=8$, $\ell\geq10^{-16}$, we have  $Z_{(2r+1)^2}([\ell,1],\ell]) \subseteq [1-10^{-15},1]$. Consequently, employing Lemma \ref{Lem:RationalIter} twice on a matrix $A=WS$ with $g(x)=\hat{Z}_{2r+1}(x;\ell)$, we see that the eigenvalues of $g(g(S))$ will be in the interval $[1-10^{-15},1]$, under the condition that all eigenvalues of $S$ are 
in $[\ell,1]$ with $\ell\geq 10^{-16}$.
In this sense, $G(G(A))\approx W$ has converged to the polar factor $W$, after two iterations of Iteration \eqref{Eq:GenIter}. Choosing a higher $r$, algorithms can be devised that converge in just one step. It was argued in \cite{NakF16} that a 2-step approach is a sensible choice to acquire a robust algorithm. This way, potential instabilities, e.g.\ in the computation of the Zolotarev coefficients $c_i$, are suppressed.

The scaled Zolotarev function can be represented in a partial fraction decomposition
\begin{align}\label{Eq:ZPartialFrac}
 &\hat{Z}_{2r+1}(x;\ell) = \hat{C}x\left(1+\sum_{j=1}^r\frac{a_j}{x^2+c_{2j-1}}\right),\\
 &a_j=-\left(\prod_{k=1}^r(c_{2j-1}-c_{2k})\right)\cdot\left(\prod_{k=1,k\neq j}^r(c_{2j-1}-c_{2k-1})\right).\label{Eq:ZPartialFrac2}
\end{align}
An iteration \eqref{Eq:GenIter} derived from \eqref{Eq:ZPartialFrac} takes the form
\begin{align}\label{Eq:ZoloIt1}
 X_{k+1}  = \hat{C}(X_k+\sum_{j=1}^r a_{j,k} X_k(X_k^{\star_M}X_k+c_{2j-1,k}I)^{-1}).
\end{align}
With $M=\Sigma$ as a signature matrix, \eqref{Eq:ZoloIt1} becomes
\begin{align}\label{Eq:ZoloIt2}
 X_{k+1} = \hat{C}(X_k+\sum_{j=1}^r a_{j,k} X_k(X_k^{*}\Sigma X+c_{2j-1,k}\Sigma)^{-1}\Sigma).
\end{align}
Computing the inverse via an $LDL^{\tran}$ decomposition leads to a first practical iteration.
\begin{align}\label{Eq:LDLZolo}
\left\{
                \begin{array}{ll}
 Z_{2j-1,k}=(X_k^*\Sigma X_k+c_{2j-1,k}\Sigma),
 \quad [L_j,D_j,P_j]=\text{ldl}(Z_{2j-1,k}),\\ \\
 X_{k+1}=\hat{C}(X_k+\sum_{j=1}^r a_j X_k P_jL_j^{-*}D_j^{-1}L_j^{-1} P_j^{\tran} \Sigma)
 \end{array}\right.
\end{align}
The first line of \eqref{Eq:LDLZolo} means that in iteration $k$, the $LDL^{\tran}$ decomposition $Z_{2j-1,k}=P_jL_jD_jL_j^*P_j^{\tran}$ is computed for each $Z_{2j-1,k}$, $j=1,\dots,r$. $P_j$ is a permutation matrix, $L_j$ is lower triangular, and $D_j$ is block-diagonal with $1\times 1$ or $2\times 2$ blocks. In \cite{BenNP22}, the special case for $r=1$ is derived. There, the iteration is rewritten such that it becomes inverse-free and a hyperbolic QR decomposition is employed instead. A special case of Theorem 5.3, in conjunction with Lemma 5.4  in \cite{BenNP22} is given in the following lemma and can be used to rewrite \eqref{Eq:ZoloIt2}.
\begin{lemma}\label{Lem:HRnoInv}
 Let $\Sigma$ be a signature matrix, $\eta\in\mathbb{R}$. For $X\in\mathbb{K}^{n\times n}$, let $\begin{bmatrix}
                                                               \eta X\\
                                                               I
                                                              \end{bmatrix}=HR$, $H=\begin{bmatrix}H_1\\H_2\end{bmatrix}\in\mathbb{K}^{2n\times n}$, $R\in\mathbb{K}^{n\times n}$ be a decomposition, such that
                                                              $H^*\begin{bmatrix}
                                                                   \Sigma&\\
                                                                   &\Sigma
                                                                  \end{bmatrix}H=\hat{\Sigma}$, where $\hat{\Sigma}\in\mathbb{R}^{n\times n}$ is another signature matrix. Then
 \begin{align*}
  \eta X(I+\eta^2 X^{\star_{\Sigma}}X)^{-1}  =  H_1\hat{\Sigma}H_2^*\Sigma.
 \end{align*}
\end{lemma}
Using Lemma \ref{Lem:HRnoInv} with $\eta=\frac{1}{\sqrt{c_{2j-1,k}}}$, \eqref{Eq:ZoloIt2} can be rewritten as

\begin{align}\label{Eq:HRZolo}
 \left\{
                \begin{array}{ll}
                  \begin{bmatrix}
                    X_k \\\sqrt{c_{2j-1,k}}I
                  \end{bmatrix} = \begin{bmatrix}H_{1,j}\\H_{2,j}\end{bmatrix}R_j,\text{ where } \begin{bmatrix}H_{1,j}\\H_{2,j}\end{bmatrix}^*\begin{bmatrix}\Sigma&\\&\Sigma\end{bmatrix}\begin{bmatrix}H_{1,j}\\H_{2,j}\end{bmatrix} = \hat{\Sigma}
\\ \\
                  X_{k+1} = \hat{C}(X_k+\sum_{j=1}^r\frac{a_j}{\sqrt{c_{2j-1}}}H_{1,j}\hat{\Sigma}H_{2,j}^*\Sigma).
                \end{array}
              \right.
\end{align}
As in iteration \eqref{Eq:LDLZolo}, the first line refers to the computation of a total of $r$ independent decompositions  $\begin{bmatrix}
                   X_k \\ \sqrt{c_{2j-1,k}}I
                  \end{bmatrix}=H_jR_j$ for $j=1, \dots, r$, per iteration step. One way of computing the needed matrix $H$ is the hyperbolic QR decomposition, which we introduced in Theorem \ref{Thm:HypQR}. Computing it via a column-elimination approach is notoriously unstable.  This is why \cite{BenNP22, BenP21a} exploit a link to the $LDL^{\tran}$ factorization and introduces the \texttt{LDLIQR2} algorithm.  
                  
                  Algorithm \ref{Alg:ZoloPD} is the pseudocode of a Zolotarev-based computation of the generalized polar factor. We assume that convergence is reached after just two steps, which are explicitly written in the Algorithm. For the computation of iterate $X_1$, iteration \eqref{Eq:HRZolo} is employed. For the computation of the $H$ matrices we use the \texttt{LDLIQR2} algorithm from \cite{BenNP22}, which showed a better numerical stability than column-elimination based approaches. The second iterate $X_2$ can safely be computed using the $LDL^{\tran}$-based iteration \eqref{Eq:LDLZolo}
                  for the same reasoning given in \cite{NakF16}. 
                  The parameter estimation and the scaling of $A$ (Steps \ref{Alg:ZoloPD:1st} and \ref{Alg:ZoloPD:Scale}) is needed to make sure that the eigenvalues of the self-adjoint factor lie in the interval $[\ell,1]$ (see Lemma \ref{Lem:RationalIter} and the discussion following).
                  In our implementation, these are bounded using the MATLAB functions \texttt{normest} and \texttt{condest}.

\begin{algorithm}[ht]
\caption{Hyperbolic Zolo-PD for definite pseudosymmetric matrices} \label{Alg:ZoloPD}
 \begin{algorithmic}[1]
  \Require Signature matrix $\Sigma$ with $p$ positive and $n-p$ negative values on the diagonal, $A\in\mathbb{C}^{n\times n}$ such that $\Sigma A$ is Hermitian positive definite.
  \Ensure $S=\sign{A}$.
  \State Estimate $\alpha\gtrsim \max\{|\lambda|:\ \lambda\in\Lambda(A)\}$, $\beta\lesssim \min\{|\lambda|:\ \lambda\in\Lambda(A)\}$. \label{Alg:ZoloPD:1st}
  \State $X_0\leftarrow \frac{1}{\alpha}A$, $\ell\leftarrow\frac{\beta}{\alpha}$. \label{Alg:ZoloPD:Scale}
  \Statex \textbf{First iteration:}
  \For{$j=1,\dots, 2r$}\label{Alg:ZoloPD:CoeffStart}
    \State $c_j\leftarrow\ell^2\text{sn}^2(\frac{\iu K'}{2r+1};\ell')     /\text{cn}^2(\frac{\iu K'}{2r+1};\ell')$.
    \Comment See \eqref{Eq:Compcis}
  \EndFor
  \For{$j=1,\dots, r$}
    \State $a_j\leftarrow-\left(\prod_{k=1}^r(c_{2j-1}-c_{2k})\right)\cdot\left(\prod_{k=1,k\neq j}^r(c_{2j-1}-c_{2k-1})\right)$.
    \Comment See \eqref{Eq:ZPartialFrac2}
  \EndFor
  \State $\hat{C}\leftarrow\prod_j^r\frac{1+c_{2j-1}}{1+c_{2j}}$ \Comment See \eqref{Eq:hatZ}\label{Alg:ZoloPD:CoeffEnd}
  \State Compute $X_1$ according to \eqref{Eq:HRZolo}, using \texttt{LDLIQR2} algorithm in \cite{BenNP22}:
  \begin{align*}
 \left\{
                \begin{array}{ll}
                  \begin{bmatrix}
                    X_0\\\sqrt{c_{2j-1}}I
                  \end{bmatrix} = \begin{bmatrix}H_{1,j}\\H_{2,j}\end{bmatrix}R_j,\text{ where } \begin{bmatrix}H_{1,j}\\H_{2,j}\end{bmatrix}^*\begin{bmatrix}\Sigma&\\&\Sigma\end{bmatrix}\begin{bmatrix}H_{1,j}\\H_{2,j}\end{bmatrix} = \hat{\Sigma}
\\ \\
                  X_{1} \leftarrow \hat{C}(X_k+\sum_{j=1}^r\frac{a_j}{\sqrt{c_{2j-1}}}H_{1,j}\hat{\Sigma}H_{2,j}^*\Sigma).
                \end{array}
              \right.
\end{align*}
\State $\ell\leftarrow \hat{C}\ell\prod_{j=1}^r(\ell^2+c_{2j})/(\ell^2+c_{2j-1})$.
\State Repeat Step \ref{Alg:ZoloPD:CoeffStart} to Step \ref{Alg:ZoloPD:CoeffEnd} to update $c_j$ for $j=1,\dots,2r$ and $a_j$ for $j=1,\dots, r$ and $\hat{C}$.
\Statex \textbf{Second iteration:}
  \State Compute $X_2$ according to \eqref{Eq:LDLZolo}:
\begin{align*}
\left\{
                \begin{array}{ll}
 Z_{2j-1,k}=(X_k^*\Sigma X_k+c_{2j-1,k}\Sigma),
 \quad [L_j,D_j,P_j]=\text{ldl}(Z_{2j-1,k}),\\ \\
 X_{k+1}\leftarrow\hat{C}(X_k+\sum_{j=1}^r a_j X_k P_jL_j^{-*}D_j^{-1}L_j^{-1} P_j^{\tran} \Sigma).
 \end{array}\right.
\end{align*}
\If{ $\|X_2-X_1\|_F/\|X_2\|_F\leq u^{1/(2r+1)}$}
\State $S\leftarrow X_2$.
\Else
\State $A\leftarrow X_2$, return to Step \ref{Alg:ZoloPD:1st}.
\EndIf
\end{algorithmic}
\end{algorithm}

Algorithm \ref{Alg:ZoloPD} converges even for badly conditioned matrices. As explained in \cite{NakF16}, for well-conditioned $A$, it is possible to skip the first iteration or choose a lower Zolotarev rank $r<8$. We choose $r$ according to Table 3.1 in \cite{NakF16}.

In exact arithmetic, the algorithm converges in 2 steps, as argued above. As a safeguard for numerical errors we adopt the stopping criterion from \cite{NakF16}, $\|X_2-X_1\|_F/\|X_2\|_F\leq u^{1/(2r+1)}$, to guarantee convergence, using the known convergence rate of $2r+1$.  We assume calculations are carried out in IEEE double precision with unit roundoff $u=2^{-53}\approx 1.1 \times 10^{-16}$.


\FloatBarrier
\section{Numerical experiments}
\label{Sec:NumRes}
In this section we apply one step of spectral divide-and-conquer (Algorithm \ref{Alg:PseuDivConq}) on definite pseudosymmetric matrices. The matrix sign function is computed by the hyperbolic Zolo-PD algorithm (Algorithm \ref{Alg:ZoloPD}), algorithms based on the $\Sigma$DWH iteration presented in \cite{BenNP22}, or a scaled Newton iteration with suboptimal scaling presented in \cite{ByeX08}. We expect these algorithms to show the same convergence properties as in the symmetric case, due to Lemma \ref{Lem:IterOnSelfAdj}. Zolo-PD should converge in 2 steps, $\Sigma$DWH in 6 steps and Newton in 9 steps. We use Algorithm \ref{Alg:HypSubspaceChol} or \ref{Alg:HypSubspaceDefLDL} to compute $(\Sigma,\hat{\Sigma})$-orthogonal subspace representations used in the spectral division.  All experiments were performed in MATLAB R2017a using double-precision arithmetic running on Ubuntu 18.04.5, using an Intel(R) Core™ i7-8550U CPU with 4 cores, 8 threads, and a clock rate of 1.80 GHz. Random matrices were generated with a seed defined by \texttt{rng(0)}.

\subsection{Random pseudosymmetric matrices}
The goal of our first numerical experiment is to determine the achieved accuracy with different methods for computing the matrix sign function.
\paragraph{Example 1}
$\Sigma$ is a signature matrix, where the diagonal values are chosen to be $1$ or $-1$ with equal probability. Given a number $\kappa=\cond{A}$, we generate real $250\times 250$ matrices as $A=\Sigma Q D Q^{\tran}$. $D$ is a diagonal matrix containing equally spaced values between 1 and $\kappa$. $Q$ is a random orthogonal matrix (\texttt{Q=orth(rand(n,n))} in MATLAB). We perform 10 runs for different randomly generated matrices and compare the backward error represented by $\|Q_+^{\tran}\Sigma A Q_-\|_{\text{F}}/\|A\|_{\text{F}}$ that is achieved by the different methods described in this work, \cite{BenNP22} and \cite{ByeX08}.

The averaged results are given in Figure \ref{Fig:BackErr_all}. All methods yield backward errors smaller than $10^{-9}$, even for badly conditioned matrices. All show a similar behavior. Hyperbolic Zolo-PD exhibits the highest backward error. Compared to the DWH-based iteration this is expected, because Zolotarev functions of higher order are used. The direct application of Zolotarev functions of high degree is known to be unstable \cite{NakF16}. In the indefinite setting, this phenomenon seems to appear sooner than in the setting described in \cite{NakF16}. The accuracy of DWH can be improved by employing permuted Lagrangian graph (PLG) bases. This way, the accuracy is comparable to a Newton approach \cite{ByeX08}.
Permuted Lagrangian graph bases can also be employed for Zolotarev iterations of higher order but go beyond the scope of this work. 

Figure \ref{Fig:Rand} displays the data of the individual runs of the same experiment. Here we see that even badly conditioned matrices often achieve a backward error of $10^{-14}$, but some outliers increase the average. Further investigations are required in order to answer the question of what backward error can be achieved for a given matrix. The red crosses denote the matrices of a given $\kappa$ for which hyperbolic Zolo-PD performed worst. We see that for the same matrices $\Sigma$DWH with $\texttt{LDLIQR2}$ and the Newton iteration also perform worse than on other matrices generated in the same way. The quality therefore seems innate to the considered matrix. When PLG bases are employed, this relation can not be observed as clearly but is still noticeable.
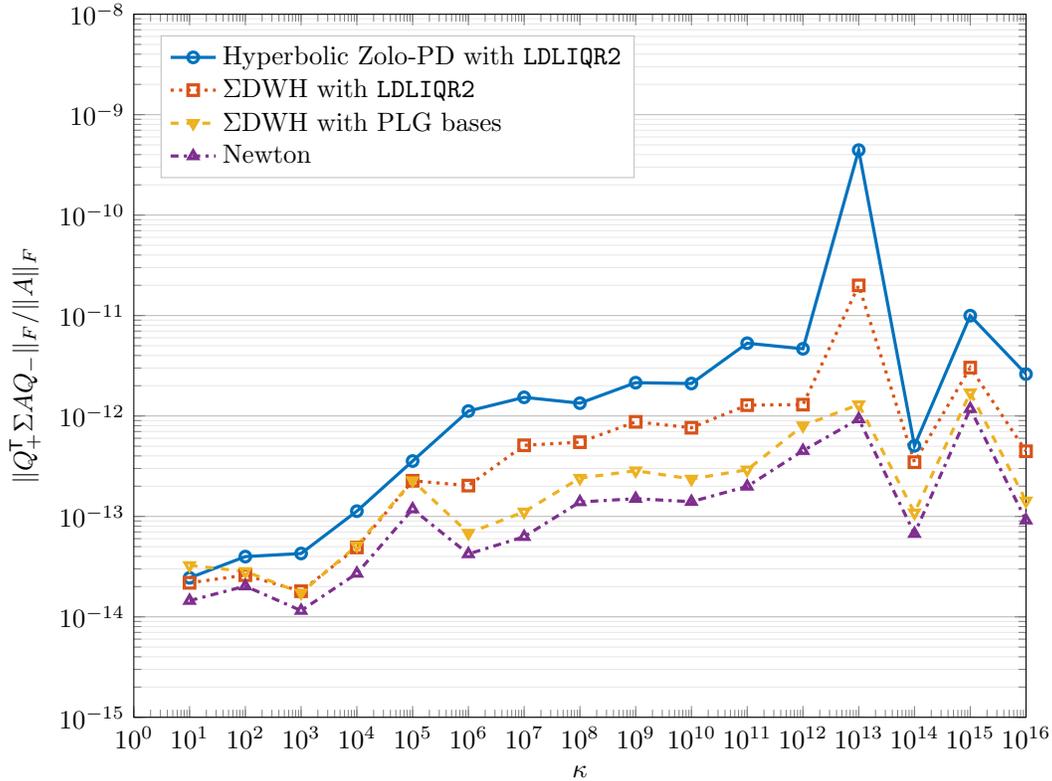
\begin{figure}[t]
\centering
 \newlength\figureheight
\newlength\figurewidth
\setlength\figureheight{0.6\textwidth}
\setlength\figurewidth{0.8\textwidth}
%
%
\definecolor{mycolor1}{rgb}{0.00000,0.44700,0.74100}%
\definecolor{mycolor2}{rgb}{0.85000,0.32500,0.09800}%
\definecolor{mycolor3}{rgb}{0.92900,0.69400,0.12500}%
\definecolor{mycolor4}{rgb}{0.49400,0.18400,0.55600}%
\begin{tikzpicture}

\begin{axis}[%
width=0.951\figurewidth,
height=\figureheight,
at={(0\figurewidth,0\figureheight)},
scale only axis,
xmin=1,
xmax=1e16,
ymode=log,
xmode=log,
ymin=1e-15,
ymax=1e-08,
ymajorgrids,
yminorgrids,
grid style={line width=.1pt, draw=gray!20},
major grid style={line width=.2pt,draw=gray!50},
yminorticks=true,
axis background/.style={fill=white},
legend style={legend cell align=left, align=left,draw=gray!50},
legend pos=north west,
every axis plot/.append style={very thick},
xlabel = $\kappa$,
ylabel = $\|Q_+^{\tran}\Sigma A Q_-\|_{F}/\|A\|_F$
]
\addplot [color=mycolor1, mark=o, mark options={solid, mycolor1}]
  table[row sep=crcr]{%
1e1	  2.44419545942173e-14\\
1e2	  3.97905623547788e-14\\
1e3	  4.27288757742412e-14\\
1e4	  1.12801882328322e-13\\
1e5	  3.55862724446711e-13\\
1e6	  1.11877191148471e-12\\
1e7	  1.53285449440623e-12\\
1e8	  1.34049484788953e-12\\
1e9	  2.14449083463424e-12\\
1e10	2.10415820466918e-12\\
1e11	5.29996649968618e-12\\
1e12	4.66395642777704e-12\\
1e13	4.43459612546108e-10\\
1e14	5.04221049545782e-13\\
1e15	9.96426307705613e-12\\
1e16	2.61172222929697e-12\\
};
\addlegendentry{Hyperbolic Zolo-PD with \texttt{LDLIQR2}}

\addplot [color=mycolor2, dotted, mark=square, mark options={solid, mycolor2}]
  table[row sep=crcr]{%
1e1	  2.18199182618386e-14\\
1e2	  2.60014557583473e-14\\
1e3	  1.79842329908799e-14\\
1e4	  4.90501933859871e-14\\
1e5	  2.26288940633217e-13\\
1e6	  2.02623609316681e-13\\
1e7	  5.10350600856577e-13\\
1e8	  5.48081847423486e-13\\
1e9	  8.73806770207915e-13\\
1e10	7.63356571415668e-13\\
1e11	1.28157812360762e-12\\
1e12	1.29598658403037e-12\\
1e13	1.99572525715005e-11\\
1e14	3.46452101634187e-13\\
1e15	3.03166844141216e-12\\
1e16	4.45267313314555e-13\\
};
\addlegendentry{$\Sigma$DWH with \texttt{LDLIQR2}}

\addplot [color=mycolor3, dashed, mark=triangle, mark options={solid, rotate=180, mycolor3}]
  table[row sep=crcr]{%
1e1	3.25047808554959e-14\\
1e2	2.83500289924541e-14\\
1e3	1.73334547181133e-14\\
1e4	5.11550962067467e-14\\
1e5	2.27393661357531e-13\\
1e6	6.82035337788212e-14\\
1e7	1.11065886426854e-13\\
1e8	2.40634745371737e-13\\
1e9	2.849892262428e-13\\
1e10	2.35921755311522e-13\\
1e11	2.91229955133149e-13\\
1e12	8.01091769142788e-13\\
1e13	1.29777139841408e-12\\
1e14	1.08817346165869e-13\\
1e15	1.70659976749395e-12\\
1e16	1.42873941648013e-13\\
};
\addlegendentry{$\Sigma$DWH with PLG bases}

\addplot [color=mycolor4, dashdotted, mark=triangle, mark options={solid, mycolor4}]
  table[row sep=crcr]{%
1e1	1.448960118106e-14\\
1e2	2.02264520202745e-14\\
1e3	1.1534384537127e-14\\
1e4	2.70310039733011e-14\\
1e5	1.1814128678454e-13\\
1e6	4.22512064757395e-14\\
1e7	6.27397691124973e-14\\
1e8	1.39140121769517e-13\\
1e9	1.50116290770914e-13\\
1e10	1.40047641324164e-13\\
1e11	1.9882875102826e-13\\
1e12	4.50789115250902e-13\\
1e13	9.31803040217272e-13\\
1e14	6.73109033304463e-14\\
1e15	1.17749330093667e-12\\
1e16	9.14752336530351e-14\\
};
\addlegendentry{Newton}

\end{axis}
\end{tikzpicture}%
 \caption{\emph{Example 1}: Average residual after one spectral divide-and-conquer step, for 10 random matrices of size $250 \times 250$ with certain condition numbers. Different methods are used for computing the matrix sign function.}\label{Fig:BackErr_all}
\end{figure}                                                  

\begin{figure}[t]
        \centering
        \begin{subfigure}[b]{0.45\textwidth}
            \centering
            \setlength\figureheight{0.525\textwidth}
\setlength\figurewidth{0.7\textwidth}
%
%
\definecolor{mycolor1}{rgb}{0.00000,0.44700,0.74100}%
\begin{tikzpicture}

\begin{axis}[%
width=0.951\figurewidth,
height=\figureheight,
at={(0\figurewidth,0\figureheight)},
scale only axis,
xmin=0,
xmax=1e16,
ymode=log,
xmode=log,
ymin=1e-15,
ymax=1e-08,
ymajorgrids,
yminorgrids,
grid style={line width=.1pt, draw=gray!20},
major grid style={line width=.2pt,draw=gray!50},
yminorticks=true,
axis background/.style={fill=white},
legend style={legend cell align=left, align=left,draw=gray!50},
legend pos=north west,
xlabel = $\kappa$,
ylabel = $\|Q_+^{\tran}\Sigma A Q_-\|_{F}/\|A\|$
]
\addplot [color=mycolor1, forget plot, very thick]
  table[row sep=crcr]{%
1e1	2.44419545942173e-14\\
1e2	3.97905623547788e-14\\
1e3	4.27288757742412e-14\\
1e4	1.12801882328322e-13\\
1e5	3.55862724446711e-13\\
1e6	1.11877191148471e-12\\
1e7	1.53285449440623e-12\\
1e8	1.34049484788953e-12\\
1e9	2.14449083463424e-12\\
1e10	2.10415820466918e-12\\
1e11	5.29996649968618e-12\\
1e12	4.66395642777704e-12\\
1e13	4.43459612546108e-10\\
1e14	5.04221049545782e-13\\
1e15	9.96426307705613e-12\\
1e16	2.61172222929697e-12\\
};
\addplot [color=gray, only marks, mark=x, mark options={solid, gray}, forget plot]
  table[row sep=crcr]{%
1e1	2.44772846996628e-14\\
1e2	1.59670430931137e-14\\
1e3	3.32974838562436e-14\\
1e4	2.58377980056676e-14\\
1e5	4.65720065030103e-13\\
1e6	7.6441003798914e-14\\
1e7	1.49860311731406e-13\\
1e8	1.02127990317693e-13\\
1e9	1.04059953060291e-13\\
1e10	2.65600848144136e-13\\
1e11	1.69660742011943e-13\\
1e12	5.75335637043844e-13\\
1e13	2.04957561807336e-14\\
1e14	5.63810747285545e-14\\
1e15	1.99579201210403e-13\\
1e16	3.20478842776474e-12\\
};
\addplot [color=gray, only marks, mark=x, mark options={solid, gray}, forget plot]
  table[row sep=crcr]{%
1e1	1.98036334345289e-14\\
1e2	7.84419087900944e-14\\
1e3	3.53662868727124e-14\\
1e4	1.23092743802938e-13\\
1e5	6.48457387085848e-13\\
1e6	6.93530943774859e-14\\
1e7	2.81221900488193e-14\\
1e8	1.32797931669118e-12\\
1e9	8.86556606717125e-14\\
1e10	1.03003698931531e-13\\
1e11	1.56473415328958e-12\\
1e12	2.83940848387238e-13\\
1e13	6.36290476745744e-13\\
1e14	2.02704879443617e-12\\
1e15	3.74147304906655e-13\\
1e16	6.20831167880506e-12\\
};
\addplot [color=gray, only marks, mark=x, mark options={solid, gray}, forget plot]
  table[row sep=crcr]{%
1e1	1.0031006582272e-14\\
1e2	3.49869772813876e-14\\
1e3	4.33029605451663e-14\\
1e4	5.08054436047042e-14\\
1e5	1.7130634738278e-13\\
1e6	4.72084349087757e-14\\
1e7	2.90687155707249e-14\\
1e8	3.87717048501721e-13\\
1e9	3.46483665756641e-14\\
1e10	4.05782977897342e-13\\
1e11	1.74863790325237e-11\\
1e12	3.46978517081858e-14\\
1e13	8.16917616026233e-14\\
1e14	3.09921327671528e-13\\
1e15	1.82463213306412e-11\\
1e16	3.4279578478717e-14\\
};
\addplot [color=gray, only marks, mark=x, mark options={solid, gray}, forget plot]
  table[row sep=crcr]{%
1e1	3.91565189795459e-14\\
1e2	2.13624033368109e-14\\
1e3	3.65923045530083e-14\\
1e4	1.54428781553527e-13\\
1e5	1.02630624714309e-13\\
1e6	1.28454827371948e-13\\
1e7	1.26445549567941e-13\\
1e8	6.99696368416055e-13\\
1e9	7.90830818635287e-14\\
1e10	4.44833002677973e-13\\
1e11	6.35894184365518e-13\\
1e12	2.63691260448126e-11\\
1e13	8.69568840520769e-10\\
1e14	7.69394629150337e-14\\
1e15	5.92498934815125e-14\\
1e16	1.43886314457444e-11\\
};
\addplot [color=gray, only marks, mark=x, mark options={solid, gray}, forget plot]
  table[row sep=crcr]{%
1e1	9.68817673199865e-15\\
1e2	1.15187730378942e-13\\
1e3	2.87782238554298e-14\\
1e4	8.09603456044973e-14\\
1e5	3.40666146490563e-13\\
1e6	8.20123505541876e-13\\
1e7	1.5346383493168e-13\\
1e8	2.45020482882949e-14\\
1e9	2.79144819960946e-12\\
1e10	2.52175562759931e-13\\
1e11	7.82331356586683e-13\\
1e12	4.92711411111119e-14\\
1e13	2.27216240304663e-12\\
1e14	4.91439973939255e-13\\
1e15	6.62251078374517e-11\\
1e16	7.27630654617415e-14\\
};
\addplot [color=gray, only marks, mark=x, mark options={solid, gray}, forget plot]
  table[row sep=crcr]{%
1e1	8.05982180207929e-15\\
1e2	2.37910050466746e-14\\
1e3	4.91744094132898e-14\\
1e4	1.75631141271208e-13\\
1e5	1.10324439528251e-13\\
1e6	8.68991257615172e-14\\
1e7	1.32146909134742e-11\\
1e8	2.17034675577269e-14\\
1e9	1.04191416956329e-13\\
1e10	5.88803170259074e-12\\
1e11	9.97946659978214e-13\\
1e12	7.06245032722372e-14\\
1e13	8.99475993085827e-14\\
1e14	3.72654770328589e-14\\
1e15	4.07075785746314e-14\\
1e16	3.91642522575926e-14\\
};
\addplot [color=gray, only marks, mark=x, mark options={solid, gray}, forget plot]
  table[row sep=crcr]{%
1e1	1.27932553131701e-14\\
1e2	1.77299541427985e-14\\
1e3	8.35649082138283e-14\\
1e4	9.65457560014004e-14\\
1e5	1.00522093249444e-13\\
1e6	8.84180370714002e-14\\
1e7	4.70832061013639e-13\\
1e8	7.0693328585716e-12\\
1e9	1.09066330872973e-11\\
1e10	6.60376881921221e-13\\
1e11	2.86941044851138e-11\\
1e12	1.40957707365961e-11\\
1e13	4.79650422194607e-13\\
1e14	1.97783427809627e-12\\
1e15	1.38531891867193e-11\\
1e16	5.66088012277188e-13\\
};
\addplot [color=gray, only marks, mark=x, mark options={solid, gray}, forget plot]
  table[row sep=crcr]{%
1e1	2.05336031282199e-14\\
1e2	1.26354586469519e-14\\
1e3	5.0579826070275e-14\\
1e4	6.36955126956254e-14\\
1e5	2.13128246040886e-14\\
1e6	2.15237451129125e-14\\
1e7	8.00791536777308e-14\\
1e8	2.50112586207059e-13\\
1e9	1.4555745765246e-13\\
1e10	2.68030341794506e-13\\
1e11	7.52316390435465e-13\\
1e12	6.81617552583679e-14\\
1e13	1.19178871774501e-13\\
1e14	2.32684512309394e-14\\
1e15	5.34447386927843e-14\\
1e16	1.4796552714891e-12\\
};
\addplot [color=gray, only marks, mark=x, mark options={solid, gray}, forget plot]
  table[row sep=crcr]{%
1e1	2.87686381019974e-14\\
1e2	1.35381701556692e-14\\
1e3	4.13590213464235e-14\\
1e4	4.29768031924831e-14\\
1e5	8.52644167356443e-13\\
1e6	3.24422925613891e-14\\
1e7	2.1023248607496e-13\\
1e8	6.77698518182411e-14\\
1e9	2.84922776331158e-12\\
1e10	8.82496751202028e-12\\
1e11	1.89849571085735e-12\\
1e12	4.9806025107463e-12\\
1e13	3.05075902688349e-13\\
1e14	1.48621349758424e-14\\
1e15	2.00936208883611e-13\\
1e16	6.16659089616696e-14\\
};
\addplot [color=gray, only marks, mark=x, mark options={solid, gray}, forget plot]
  table[row sep=crcr]{%
1e1	7.11076071686985e-14\\
1e2	6.4264972675345e-14\\
1e3	2.52733330160347e-14\\
1e4	3.14044497551164e-13\\
1e5	7.45043149025285e-13\\
1e6	9.81685504834092e-12\\
1e7	8.65749727971237e-13\\
1e8	3.45400694252575e-12\\
1e9	4.34140335934407e-12\\
1e10	3.92877951795414e-12\\
1e11	1.78022816996385e-14\\
1e12	1.12033248834478e-13\\
1e13	3.56102279174677e-09\\
1e14	2.72495204313664e-14\\
1e15	3.89947489999521e-13\\
1e16	6.18746517295344e-14\\
};
\addplot [color=red, only marks, mark=x, mark options={solid, red}, forget plot]
  table[row sep=crcr]{%
1e1	7.11076071686985e-14\\
1e2	1.15187730378942e-13\\
1e3	8.35649082138283e-14\\
1e4	3.14044497551164e-13\\
1e5	8.52644167356443e-13\\
1e6	9.81685504834092e-12\\
1e7	1.32146909134742e-11\\
1e8	7.0693328585716e-12\\
1e9	1.09066330872973e-11\\
1e10	8.82496751202028e-12\\
1e11	2.86941044851138e-11\\
1e12	2.63691260448126e-11\\
1e13	3.56102279174677e-09\\
1e14	2.02704879443617e-12\\
1e15	6.62251078374517e-11\\
1e16	1.43886314457444e-11\\
};
\end{axis}
\end{tikzpicture}%
            \caption{ Hyperbolic Zolo-PD}  
            \label{fig:ZolRand}
        \end{subfigure}
	\hspace{1em}
        \begin{subfigure}[b]{0.45\textwidth}  
            \centering 
            \setlength\figureheight{0.525\textwidth}
\setlength\figurewidth{0.7\textwidth}
%
%
\definecolor{mycolor1}{rgb}{0.85000,0.32500,0.09800}%
\begin{tikzpicture}

\begin{axis}[%
width=0.951\figurewidth,
height=\figureheight,
at={(0\figurewidth,0\figureheight)},
scale only axis,
xmin=0,
xmax=1e16,
ymode=log,
xmode=log,
ymin=1e-15,
ymax=1e-08,
ymajorgrids,
yminorgrids,
grid style={line width=.1pt, draw=gray!20},
major grid style={line width=.2pt,draw=gray!50},
yminorticks=true,
axis background/.style={fill=white},
legend style={legend cell align=left, align=left,draw=gray!50},
legend pos=north west,
xlabel = $\kappa$,
ylabel = $\|Q_+^{\tran}\Sigma A Q_-\|_{F}/\|A\|$
]
\addplot [color=mycolor1, dashed, forget plot, very thick]
  table[row sep=crcr]{%
1e1	2.18199182618386e-14\\
1e2	2.60014557583473e-14\\
1e3	1.79842329908799e-14\\
1e4	4.90501933859871e-14\\
1e5	2.26288940633217e-13\\
1e6	2.02623609316681e-13\\
1e7	5.10350600856577e-13\\
1e8	5.48081847423486e-13\\
1e9	8.73806770207915e-13\\
1e10	7.63356571415668e-13\\
1e11	1.28157812360762e-12\\
1e12	1.29598658403037e-12\\
1e13	1.99572525715005e-11\\
1e14	3.46452101634187e-13\\
1e15	3.03166844141216e-12\\
1e16	4.45267313314555e-13\\
};
\addplot [color=gray, only marks, mark=x, mark options={solid, gray}, forget plot]
  table[row sep=crcr]{%
1e1	2.39128109419803e-14\\
1e2	9.07678270738782e-15\\
1e3	1.97356102770054e-14\\
1e4	2.03076744135627e-14\\
1e5	3.79580340111838e-13\\
1e6	6.6214023651833e-14\\
1e7	8.40623561993001e-14\\
1e8	1.91594924507775e-14\\
1e9	7.39251002768383e-14\\
1e10	1.61791769015534e-13\\
1e11	5.38253799221877e-14\\
1e12	5.14701922196301e-13\\
1e13	1.65570525633127e-14\\
1e14	4.04026638848029e-14\\
1e15	4.96341387778217e-14\\
1e16	6.85583128608413e-13\\
};
\addplot [color=gray, only marks, mark=x, mark options={solid, gray}, forget plot]
  table[row sep=crcr]{%
1e1	1.78580129780281e-14\\
1e2	6.06723237115426e-14\\
1e3	1.47581560284382e-14\\
1e4	4.3103853342068e-14\\
1e5	5.87168580633389e-13\\
1e6	1.93919940140375e-14\\
1e7	1.6741838794989e-14\\
1e8	9.80176646040805e-13\\
1e9	4.6119142060946e-14\\
1e10	9.76557022478693e-14\\
1e11	5.58242898484979e-13\\
1e12	7.6244469816142e-14\\
1e13	3.79344256085225e-13\\
1e14	1.34439247563775e-12\\
1e15	1.44674279947905e-13\\
1e16	1.02808829233702e-12\\
};
\addplot [color=gray, only marks, mark=x, mark options={solid, gray}, forget plot]
  table[row sep=crcr]{%
1e1	1.09122255414301e-14\\
1e2	2.33665155455426e-14\\
1e3	1.2246468406361e-14\\
1e4	3.89180786698546e-14\\
1e5	8.99252866753757e-14\\
1e6	3.02181805019868e-14\\
1e7	2.34930173817881e-14\\
1e8	1.1830917923817e-13\\
1e9	3.21637043630541e-14\\
1e10	2.34203699599705e-13\\
1e11	6.29260503901761e-12\\
1e12	2.32966714576468e-14\\
1e13	6.12960890766468e-14\\
1e14	1.42922986107614e-13\\
1e15	2.24948803136772e-12\\
1e16	1.98250899266169e-14\\
};
\addplot [color=gray, only marks, mark=x, mark options={solid, gray}, forget plot]
  table[row sep=crcr]{%
1e1	3.99433136189025e-14\\
1e2	9.87710089224855e-15\\
1e3	2.18705112905233e-14\\
1e4	6.30791996685904e-14\\
1e5	2.52945222960883e-14\\
1e6	9.35388966823932e-14\\
1e7	1.33234849051789e-13\\
1e8	2.03517128087917e-13\\
1e9	6.21648731000999e-14\\
1e10	1.44512171036415e-13\\
1e11	1.40709421797573e-13\\
1e12	4.52714163935779e-12\\
1e13	6.97062951476614e-11\\
1e14	3.94367517450279e-14\\
1e15	4.19556485305036e-14\\
1e16	2.06108561973784e-12\\
};
\addplot [color=gray, only marks, mark=x, mark options={solid, gray}, forget plot]
  table[row sep=crcr]{%
1e1	1.16403542571965e-14\\
1e2	8.982337916496e-14\\
1e3	1.26241532508753e-14\\
1e4	6.01584606092128e-14\\
1e5	3.57648618968149e-13\\
1e6	2.42100785615893e-13\\
1e7	1.11196079930522e-13\\
1e8	2.27640427791538e-14\\
1e9	1.33537040892791e-12\\
1e10	7.16621778574908e-14\\
1e11	3.86949481402431e-13\\
1e12	2.59997460081534e-14\\
1e13	5.69098682853416e-13\\
1e14	2.40576579883443e-13\\
1e15	2.42116638582156e-11\\
1e16	6.10244242123642e-14\\
};
\addplot [color=gray, only marks, mark=x, mark options={solid, gray}, forget plot]
  table[row sep=crcr]{%
1e1	1.01932100611712e-14\\
1e2	1.61344516514239e-14\\
1e3	2.17910960501657e-14\\
1e4	3.63760486377881e-14\\
1e5	7.80769458900578e-14\\
1e6	4.24669617291075e-14\\
1e7	3.99431211958761e-12\\
1e8	1.53421768996072e-14\\
1e9	7.18703788853884e-14\\
1e10	3.16403413658167e-12\\
1e11	4.43275975005064e-13\\
1e12	1.95994722909712e-14\\
1e13	6.36975740535051e-14\\
1e14	2.15594582503782e-14\\
1e15	1.95779012781267e-14\\
1e16	2.3706056392278e-14\\
};
\addplot [color=gray, only marks, mark=x, mark options={solid, gray}, forget plot]
  table[row sep=crcr]{%
1e1	1.00640684897317e-14\\
1e2	1.23216253809079e-14\\
1e3	1.7825796211859e-14\\
1e4	5.48370156734493e-14\\
1e5	5.45426159760949e-14\\
1e6	5.84498799037775e-14\\
1e7	3.12008566116092e-13\\
1e8	2.64778993882839e-12\\
1e9	5.36849088554613e-12\\
1e10	2.99426666568913e-13\\
1e11	4.06392522561006e-12\\
1e12	4.20195714338536e-12\\
1e13	2.53463104129941e-13\\
1e14	1.59486866382595e-12\\
1e15	3.25226144126872e-12\\
1e16	1.42955288582942e-13\\
};
\addplot [color=gray, only marks, mark=x, mark options={solid, gray}, forget plot]
  table[row sep=crcr]{%
1e1	2.74251784690874e-14\\
1e2	8.29630345092827e-15\\
1e3	1.69400814302533e-14\\
1e4	3.37788664107511e-14\\
1e5	2.18794885394479e-14\\
1e6	1.50421139390909e-14\\
1e7	3.67987536932726e-14\\
1e8	1.35341168981657e-13\\
1e9	8.80760382743651e-14\\
1e10	2.8930997172066e-13\\
1e11	3.60924678347638e-13\\
1e12	6.37513810534471e-14\\
1e13	1.1042345381775e-13\\
1e14	1.42936383273315e-14\\
1e15	1.83411979340539e-14\\
1e16	3.65521479429917e-13\\
};
\addplot [color=gray, only marks, mark=x, mark options={solid, gray}, forget plot]
  table[row sep=crcr]{%
1e1	2.12210880189762e-14\\
1e2	1.19368943374702e-14\\
1e3	2.44622436870047e-14\\
1e4	2.96664242043227e-14\\
1e5	1.70493635428497e-13\\
1e6	2.15448132134095e-14\\
1e7	1.18455305821154e-13\\
1e8	2.03603209789816e-14\\
1e9	5.81138903513366e-13\\
1e10	2.17238909063467e-12\\
1e11	5.06934428044944e-13\\
1e12	3.48426916802892e-12\\
1e13	1.8052970500549e-13\\
1e14	1.31152519053864e-14\\
1e15	1.41761278540545e-13\\
1e16	2.81454852829968e-14\\
};
\addplot [color=gray, only marks, mark=x, mark options={solid, gray}, forget plot]
  table[row sep=crcr]{%
1e1	4.50289202418823e-14\\
1e2	1.85091807410616e-14\\
1e3	1.75882132763132e-14\\
1e4	1.10276312230271e-13\\
1e5	4.98279371813236e-13\\
1e6	1.43726844391528e-12\\
1e7	2.73203121989257e-13\\
1e8	1.3180583799494e-12\\
1e9	1.07874826713105e-12\\
1e10	9.98580328893747e-13\\
1e11	8.38870844374331e-15\\
1e12	2.2904226708942e-14\\
1e13	1.28231820649759e-10\\
1e14	1.29525467741776e-14\\
1e15	1.87326638260613e-13\\
1e16	3.67382686351594e-14\\
};
\addplot [color=red, only marks, mark=x, mark options={solid, red}, forget plot]
  table[row sep=crcr]{%
1e1	4.50289202418823e-14\\
1e2	8.982337916496e-14\\
1e3	1.7825796211859e-14\\
1e4	1.10276312230271e-13\\
1e5	1.70493635428497e-13\\
1e6	1.43726844391528e-12\\
1e7	3.99431211958761e-12\\
1e8	2.64778993882839e-12\\
1e9	5.36849088554613e-12\\
1e10	2.17238909063467e-12\\
1e11	4.06392522561006e-12\\
1e12	4.52714163935779e-12\\
1e13	1.28231820649759e-10\\
1e14	1.34439247563775e-12\\
1e15	2.42116638582156e-11\\
1e16	2.06108561973784e-12\\
};
\end{axis}
\end{tikzpicture}%
            \caption%
            {{ $\Sigma$DWH with \texttt{LDLIQR2}}}    
            \label{fig:DWHRand}
        \end{subfigure}
        \vskip\baselineskip
        \begin{subfigure}[b]{0.45\textwidth}   
            \centering 
            \setlength\figureheight{0.525\textwidth}
\setlength\figurewidth{0.7\textwidth}
%
%
\definecolor{mycolor1}{rgb}{0.92900,0.69400,0.12500}%
\begin{tikzpicture}

\begin{axis}[%
width=0.951\figurewidth,
height=\figureheight,
at={(0\figurewidth,0\figureheight)},
scale only axis,
xmin=1,
xmax=1e16,
ymode=log,
xmode=log,
ymin=1e-15,
ymax=1e-08,
ymajorgrids,
yminorgrids,
grid style={line width=.1pt, draw=gray!20},
major grid style={line width=.2pt,draw=gray!50},
yminorticks=true,
axis background/.style={fill=white},
legend style={legend cell align=left, align=left,draw=gray!50},
legend pos=north west,
xlabel = $\kappa$,
ylabel = $\|Q_+^{\tran}\Sigma A Q_-\|_{F}/\|A\|$
]
\addplot [color=mycolor1, dashed, forget plot, very thick]
  table[row sep=crcr]{%
1e1	3.25047808554959e-14\\
1e2	2.83500289924541e-14\\
1e3	1.73334547181133e-14\\
1e4	5.11550962067467e-14\\
1e5	2.27393661357531e-13\\
1e6	6.82035337788212e-14\\
1e7	1.11065886426854e-13\\
1e8	2.40634745371737e-13\\
1e9	2.849892262428e-13\\
1e10	2.35921755311522e-13\\
1e11	2.91229955133149e-13\\
1e12	8.01091769142788e-13\\
1e13	1.29777139841408e-12\\
1e14	1.08817346165869e-13\\
1e15	1.70659976749395e-12\\
1e16	1.42873941648013e-13\\
};
\addplot [color=gray, only marks, mark=x, mark options={solid, gray}, forget plot]
  table[row sep=crcr]{%
1e1	4.88009619655508e-14\\
1e2	1.03034161485787e-14\\
1e3	1.71508161642274e-14\\
1e4	1.78064648377514e-14\\
1e5	5.99124686160394e-13\\
1e6	6.84748303927205e-14\\
1e7	8.42391213157705e-14\\
1e8	2.37906763016411e-14\\
1e9	6.3762175447681e-14\\
1e10	7.0128023074701e-14\\
1e11	4.57624195856344e-14\\
1e12	3.05004879410692e-13\\
1e13	1.54960763322228e-14\\
1e14	4.29054109986228e-14\\
1e15	4.47965882871395e-14\\
1e16	1.54070451362466e-13\\
};
\addplot [color=gray, only marks, mark=x, mark options={solid, gray}, forget plot]
  table[row sep=crcr]{%
1e1	2.27389245400208e-14\\
1e2	5.02453778652101e-14\\
1e3	1.74505657881403e-14\\
1e4	4.37677944166761e-14\\
1e5	6.19344003508329e-13\\
1e6	2.1557423563851e-14\\
1e7	1.59398027022993e-14\\
1e8	3.48185850733098e-13\\
1e9	5.29371711222075e-14\\
1e10	7.02071885978631e-14\\
1e11	2.13235537238568e-13\\
1e12	6.64158842207559e-14\\
1e13	4.07963026255592e-13\\
1e14	1.92730480899007e-13\\
1e15	9.35294410820617e-14\\
1e16	2.4053880902436e-13\\
};
\addplot [color=gray, only marks, mark=x, mark options={solid, gray}, forget plot]
  table[row sep=crcr]{%
1e1	1.52152641999249e-14\\
1e2	3.75387071237551e-14\\
1e3	1.51464750042708e-14\\
1e4	2.79619715761701e-14\\
1e5	8.24046189246636e-14\\
1e6	3.08814042510767e-14\\
1e7	3.2305889652887e-14\\
1e8	8.21116553568934e-14\\
1e9	2.98032607309103e-14\\
1e10	8.58285747723027e-14\\
1e11	1.2572617523887e-12\\
1e12	1.67536867541742e-14\\
1e13	4.93798757755282e-14\\
1e14	5.81486574235109e-14\\
1e15	7.97780230385018e-13\\
1e16	2.17122425163587e-14\\
};
\addplot [color=gray, only marks, mark=x, mark options={solid, gray}, forget plot]
  table[row sep=crcr]{%
1e1	4.91529072927878e-14\\
1e2	1.17997882626236e-14\\
1e3	1.80845560633568e-14\\
1e4	7.97699999684861e-14\\
1e5	2.30927418512695e-14\\
1e6	1.0198194463084e-13\\
1e7	9.46461027862377e-14\\
1e8	1.86952499961654e-13\\
1e9	5.63120761812989e-14\\
1e10	1.28468879832855e-13\\
1e11	1.11448726773614e-13\\
1e12	3.05751287534762e-13\\
1e13	8.24022199258467e-12\\
1e14	3.55537599254263e-14\\
1e15	5.29243627620495e-14\\
1e16	5.80537294290531e-13\\
};
\addplot [color=gray, only marks, mark=x, mark options={solid, gray}, forget plot]
  table[row sep=crcr]{%
1	1.53509943029077e-14\\
2	9.03285616356152e-14\\
3	1.34110614886419e-14\\
4	5.73471060088987e-14\\
5	3.40808265210058e-13\\
6	1.87079744098317e-13\\
7	7.44967483433976e-14\\
8	2.62710041022709e-14\\
9	6.64145382940876e-13\\
10	6.8370129817173e-14\\
11	3.17680335542466e-13\\
12	2.49480093690378e-14\\
13	1.82660370284221e-13\\
14	8.61548993880635e-14\\
15	1.39858679038958e-11\\
16	3.4884091354054e-14\\
};
\addplot [color=gray, only marks, mark=x, mark options={solid, gray}, forget plot]
  table[row sep=crcr]{%
1e1	1.20562827992122e-14\\
1e2	1.99272092323714e-14\\
1e3	1.55869989393329e-14\\
1e4	3.53869331289418e-14\\
1e5	5.26553355830777e-14\\
1e6	4.89803423668628e-14\\
1e7	2.01425336396679e-13\\
1e8	1.8184490304092e-14\\
1e9	6.44241580921118e-14\\
1e10	7.66913802440022e-13\\
1e11	1.65783544403179e-13\\
1e12	1.94217740578294e-14\\
1e13	4.88970493878097e-14\\
1e14	2.81614106482595e-14\\
1e15	2.62135258680999e-14\\
1e16	2.47469336173079e-14\\
};
\addplot [color=gray, only marks, mark=x, mark options={solid, gray}, forget plot]
  table[row sep=crcr]{%
1e1	1.67305462146146e-14\\
1e2	1.30901428937465e-14\\
1e3	1.44337692428179e-14\\
1e4	6.36004222321045e-14\\
1e5	3.37372999686801e-14\\
1e6	6.49753471643591e-14\\
1e7	2.92772006853301e-13\\
1e8	3.71026776268581e-13\\
1e9	1.47765745384468e-12\\
1e10	1.8178757157713e-13\\
1e11	5.02492100081865e-13\\
1e12	4.40487270513468e-12\\
1e13	2.19478005587862e-13\\
1e14	5.96520912042647e-13\\
1e15	1.73103102551984e-12\\
1e16	1.6606913006063e-13\\
};
\addplot [color=gray, only marks, mark=x, mark options={solid, gray}, forget plot]
  table[row sep=crcr]{%
1e1	2.73984223976267e-14\\
1e2	8.07891008102211e-15\\
1e3	1.89208665025211e-14\\
1e4	2.92140146941418e-14\\
1e5	1.62835484690329e-14\\
1e6	1.81986472942335e-14\\
1e7	3.92303414795777e-14\\
1e8	1.51433600667272e-13\\
1e9	1.04082634491386e-13\\
1e10	2.32608535287373e-13\\
1e11	1.62473163703645e-13\\
1e12	3.47415050866228e-14\\
1e13	8.33326486192434e-14\\
1e14	1.56142061750171e-14\\
1e15	2.27758892527609e-14\\
1e16	1.48722838717649e-13\\
};
\addplot [color=gray, only marks, mark=x, mark options={solid, gray}, forget plot]
  table[row sep=crcr]{%
1e1	3.48489858777978e-14\\
1e2	1.3919139812138e-14\\
1e3	2.43283891756103e-14\\
1e4	4.52188618480247e-14\\
1e5	1.19876835443933e-13\\
1e6	2.03868625183384e-14\\
1e7	8.76425346507936e-14\\
1e8	2.48978089832104e-14\\
1e9	1.39825934595684e-13\\
1e10	5.14811575938604e-13\\
1e11	1.27906147952542e-13\\
1e12	2.81293613379237e-12\\
1e13	7.43495745967773e-14\\
1e14	1.38397810292387e-14\\
1e15	1.12952408885526e-13\\
1e16	3.60140574461461e-14\\
};
\addplot [color=gray, only marks, mark=x, mark options={solid, gray}, forget plot]
  table[row sep=crcr]{%
1e1	8.27545189645153e-14\\
1e2	2.82690368694798e-14\\
1e3	1.88210488122137e-14\\
1e4	1.11477393356272e-13\\
1e5	3.86609278455871e-13\\
1e6	1.19518791507613e-13\\
1e7	1.87960980087598e-13\\
1e8	1.17349309103866e-12\\
1e9	1.96942014981164e-13\\
1e10	2.40093271777192e-13\\
1e11	8.25582366127994e-15\\
1e12	2.0071826066953e-14\\
1e13	3.65593536471688e-12\\
1e14	1.85439431289e-14\\
1e15	1.98126299001179e-13\\
1e16	2.14435680906291e-14\\
};
\addplot [color=red, only marks, mark=x, mark options={solid, red}, forget plot]
  table[row sep=crcr]{%
1e1	8.27545189645153e-14\\
1e2	9.03285616356152e-14\\
1e3	1.44337692428179e-14\\
1e4	1.11477393356272e-13\\
1e5	1.19876835443933e-13\\
1e6	1.19518791507613e-13\\
1e7	2.01425336396679e-13\\
1e8	3.71026776268581e-13\\
1e9	1.47765745384468e-12\\
1e10	5.14811575938604e-13\\
1e11	5.02492100081865e-13\\
1e12	3.05751287534762e-13\\
1e13	3.65593536471688e-12\\
1e14	1.92730480899007e-13\\
1e15	1.39858679038958e-11\\
1e16	5.80537294290531e-13\\
};
\end{axis}
\end{tikzpicture}%
            \caption
            {{ $\Sigma$DWH with PLG bases}}    
            \label{fig:DWHSubsRand}
        \end{subfigure}
        \hspace{1em}
        \begin{subfigure}[b]{0.45\textwidth}   
            \centering 
            \setlength\figureheight{0.525\textwidth}
\setlength\figurewidth{0.7\textwidth}
%
%
\definecolor{mycolor1}{rgb}{0.49400,0.18400,0.55600}%
\begin{tikzpicture}

\begin{axis}[%
width=0.951\figurewidth,
height=\figureheight,
at={(0\figurewidth,0\figureheight)},
scale only axis,
xmin=0,
xmax=1e16,
ymode=log,
xmode=log,
ymin=1e-15,
ymax=1e-08,
ymajorgrids,
yminorgrids,
grid style={line width=.1pt, draw=gray!20},
major grid style={line width=.2pt,draw=gray!50},
yminorticks=true,
axis background/.style={fill=white},
legend style={legend cell align=left, align=left,draw=gray!50},
legend pos=north west,
xlabel = $\kappa$,
ylabel = $\|Q_+^{\tran}\Sigma A Q_-\|_{F}/\|A\|$
]
\addplot [color=mycolor1, dashdotted, forget plot, very thick]
 table[row sep=crcr]{%
1e1	1.448960118106e-14\\
1e2	2.02264520202745e-14\\
1e3	1.1534384537127e-14\\
1e4	2.70310039733011e-14\\
1e5	1.1814128678454e-13\\
1e6	4.22512064757395e-14\\
1e7	6.27397691124973e-14\\
1e8	1.39140121769517e-13\\
1e9	1.50116290770914e-13\\
1e10	1.40047641324164e-13\\
1e11	1.9882875102826e-13\\
1e12	4.50789115250902e-13\\
1e13	9.31803040217272e-13\\
1e14	6.73109033304463e-14\\
1e15	1.17749330093667e-12\\
1e16	9.14752336530351e-14\\
};
\addplot [color=gray, only marks, mark=x, mark options={solid, gray}, forget plot]
  table[row sep=crcr]{%
1e1	1.85705232509419e-14\\
1e2	6.80172239848425e-15\\
1e3	1.38493390486599e-14\\
1e4	1.05818941984963e-14\\
1e5	2.25845444173964e-13\\
1e6	3.56791552708862e-14\\
1e7	4.80611714774733e-14\\
1e8	1.15441532320485e-14\\
1e9	3.33823688166197e-14\\
1e10	5.72537008627348e-14\\
1e11	2.78978846796502e-14\\
1e12	1.94474080949949e-13\\
1e13	8.46118032331273e-15\\
1e14	2.60985791839476e-14\\
1e15	2.79614751806634e-14\\
1e16	1.18187122400579e-13\\
};
\addplot [color=gray, only marks, mark=x, mark options={solid, gray}, forget plot]
  table[row sep=crcr]{%
1e1	1.10604488074839e-14\\
1e2	4.18990780937951e-14\\
1e3	1.03935680571458e-14\\
1e4	2.84123828473531e-14\\
1e5	2.88750742152259e-13\\
1e6	1.36919965556715e-14\\
1e7	9.1497994403248e-15\\
1e8	1.51568415166569e-13\\
1e9	2.70829875150659e-14\\
1e10	3.08145490996877e-14\\
1e11	1.33298702835514e-13\\
1e12	3.13018223581848e-14\\
1e13	2.45845071079901e-13\\
1e14	1.37004855535808e-13\\
1e15	5.61120679732334e-14\\
1e16	1.87423889937411e-13\\
};
\addplot [color=gray, only marks, mark=x, mark options={solid, gray}, forget plot]
  table[row sep=crcr]{%
1e1	7.07848206650932e-15\\
1e2	2.11973364148683e-14\\
1e3	8.7838495357917e-15\\
1e4	1.86697085184883e-14\\
1e5	4.84610050555491e-14\\
1e6	1.55762990205371e-14\\
1e7	1.48521034488659e-14\\
1e8	4.22889911682458e-14\\
1e9	2.04300146241264e-14\\
1e10	5.31953783950132e-14\\
1e11	9.75646881390043e-13\\
1e12	1.14679075922655e-14\\
1e13	2.64502509324397e-14\\
1e14	3.31648294917687e-14\\
1e15	6.17155801202303e-13\\
1e16	1.12115076628726e-14\\
};
\addplot [color=gray, only marks, mark=x, mark options={solid, gray}, forget plot]
  table[row sep=crcr]{%
1e1	2.21351189130293e-14\\
1e2	9.82724704714219e-15\\
1e3	1.25138965871266e-14\\
1e4	4.32791366585283e-14\\
1e5	1.4031997690514e-14\\
1e6	5.18275617394574e-14\\
1e7	6.11574525074353e-14\\
1e8	9.5166694166993e-14\\
1e9	2.90810422603173e-14\\
1e10	7.89511767790396e-14\\
1e11	6.64978086196674e-14\\
1e12	3.0645125032325e-13\\
1e13	5.0469359787811e-12\\
1e14	1.86561991879149e-14\\
1e15	2.23439292519952e-14\\
1e16	3.68961744572954e-13\\
};
\addplot [color=gray, only marks, mark=x, mark options={solid, gray}, forget plot]
  table[row sep=crcr]{%
1e1	7.55659914206604e-15\\
1e2	6.47606612012045e-14\\
1e3	8.46858225709815e-15\\
1e4	3.21914648425408e-14\\
1e5	2.00240282205088e-13\\
1e6	9.86420955158792e-14\\
1e7	3.6985480322931e-14\\
1e8	1.01917432943981e-14\\
1e9	4.71203184318095e-13\\
1e10	3.6185943201089e-14\\
1e11	1.84649544228231e-13\\
1e12	1.88289411220844e-14\\
1e13	1.26791854530393e-13\\
1e14	4.90178687140408e-14\\
1e15	9.84708450538673e-12\\
1e16	2.00355255505011e-14\\
};
\addplot [color=gray, only marks, mark=x, mark options={solid, gray}, forget plot]
  table[row sep=crcr]{%
1e1	6.30920020962244e-15\\
1e2	1.31000199065952e-14\\
1e3	9.5983852946362e-15\\
1e4	1.914849629957e-14\\
1e5	3.36552504966507e-14\\
1e6	2.46049898429871e-14\\
1e7	1.30163987115487e-13\\
1e8	8.41366465123873e-15\\
1e9	3.855982495587e-14\\
1e10	3.85605857150385e-13\\
1e11	1.00648479291632e-13\\
1e12	1.48679246116066e-14\\
1e13	2.84212920708693e-14\\
1e14	1.22086305878091e-14\\
1e15	1.23488385609767e-14\\
1e16	1.19748404769869e-14\\
};
\addplot [color=gray, only marks, mark=x, mark options={solid, gray}, forget plot]
  table[row sep=crcr]{%
1e1	7.47768725848634e-15\\
1e2	1.07545870597461e-14\\
1e3	1.04472483716858e-14\\
1e4	2.81834581052569e-14\\
1e5	1.8638746850259e-14\\
1e6	3.82440673239018e-14\\
1e7	1.50221822307736e-13\\
1e8	2.02300095809268e-13\\
1e9	5.79924085805162e-13\\
1e10	7.85548841873427e-14\\
1e11	3.42970588681543e-13\\
1e12	2.56074853451901e-12\\
1e13	1.55909019192531e-13\\
1e14	3.72893747865862e-13\\
1e15	1.01064665664292e-12\\
1e16	7.03960453265637e-14\\
};
\addplot [color=gray, only marks, mark=x, mark options={solid, gray}, forget plot]
  table[row sep=crcr]{%
1e1	1.17333238925591e-14\\
1e2	5.75157358619307e-15\\
1e3	1.51281879575227e-14\\
1e4	1.81173453410822e-14\\
1e5	9.14311581180332e-15\\
1e6	8.09718145912691e-15\\
1e7	2.30318801264867e-14\\
1e8	8.06565927793249e-14\\
1e9	7.12761729949279e-14\\
1e10	1.52699465622753e-13\\
1e11	7.77751289050468e-14\\
1e12	1.91571097800451e-14\\
1e13	5.71227578657165e-14\\
1e14	8.99230062880876e-15\\
1e15	1.30989617347924e-14\\
1e16	9.17751867828341e-14\\
};
\addplot [color=gray, only marks, mark=x, mark options={solid, gray}, forget plot]
  table[row sep=crcr]{%
1e1	1.80774158823566e-14\\
1e2	7.35059518103667e-15\\
1e3	1.43968061058572e-14\\
1e4	1.66711195993248e-14\\
1e5	8.06959215098305e-14\\
1e6	1.20096708700178e-14\\
1e7	4.26344621028538e-14\\
1e8	1.13702262834769e-14\\
1e9	8.57662379602391e-14\\
1e10	3.58894954886138e-13\\
1e11	7.31437007253826e-14\\
1e12	1.33927903316248e-12\\
1e13	4.8330748638976e-14\\
1e14	6.97466479217522e-15\\
1e15	6.77975403157791e-14\\
1e16	2.18415354173232e-14\\
};
\addplot [color=gray, only marks, mark=x, mark options={solid, gray}, forget plot]
  table[row sep=crcr]{%
1e1	3.48972123875449e-14\\
1e2	2.08216993136794e-14\\
1e3	1.17639821557462e-14\\
1e4	5.50550333223698e-14\\
1e5	2.6195036189948e-13\\
1e6	1.2413904715893e-13\\
1e7	1.11139532275379e-13\\
1e8	7.77900641143607e-13\\
1e9	1.44456988458719e-13\\
1e10	1.68320503057461e-13\\
1e11	5.75879092588718e-15\\
1e12	1.13145480901467e-14\\
1e13	3.57376224875748e-12\\
1e14	8.09735731632819e-15\\
1e15	1.00383233117317e-13\\
1e16	1.29449384023258e-14\\
};
\addplot [color=red, only marks, mark=x, mark options={solid, red}, forget plot]
  table[row sep=crcr]{%
1e1	3.48972123875449e-14\\
1e2	6.47606612012045e-14\\
1e3	1.04472483716858e-14\\
1e4	5.50550333223698e-14\\
1e5	8.06959215098305e-14\\
1e6	1.2413904715893e-13\\
1e7	1.30163987115487e-13\\
1e8	2.02300095809268e-13\\
1e9	5.79924085805162e-13\\
1e10	3.58894954886138e-13\\
1e11	3.42970588681543e-13\\
1e12	3.0645125032325e-13\\
1e13	3.57376224875748e-12\\
1e14	1.37004855535808e-13\\
1e15	9.84708450538673e-12\\
1e16	3.68961744572954e-13\\
};
\end{axis}
\end{tikzpicture}%
            \caption[]%
            {{ Newton}}    
            \label{fig:NewtRand}
        \end{subfigure}
        \caption{\emph{Example 1}: Residuals after one step of spectral divide-and-conquer for 10 runs with randomly generated matrices of certain condition numbers.} 
        \label{Fig:Rand}
    \end{figure}
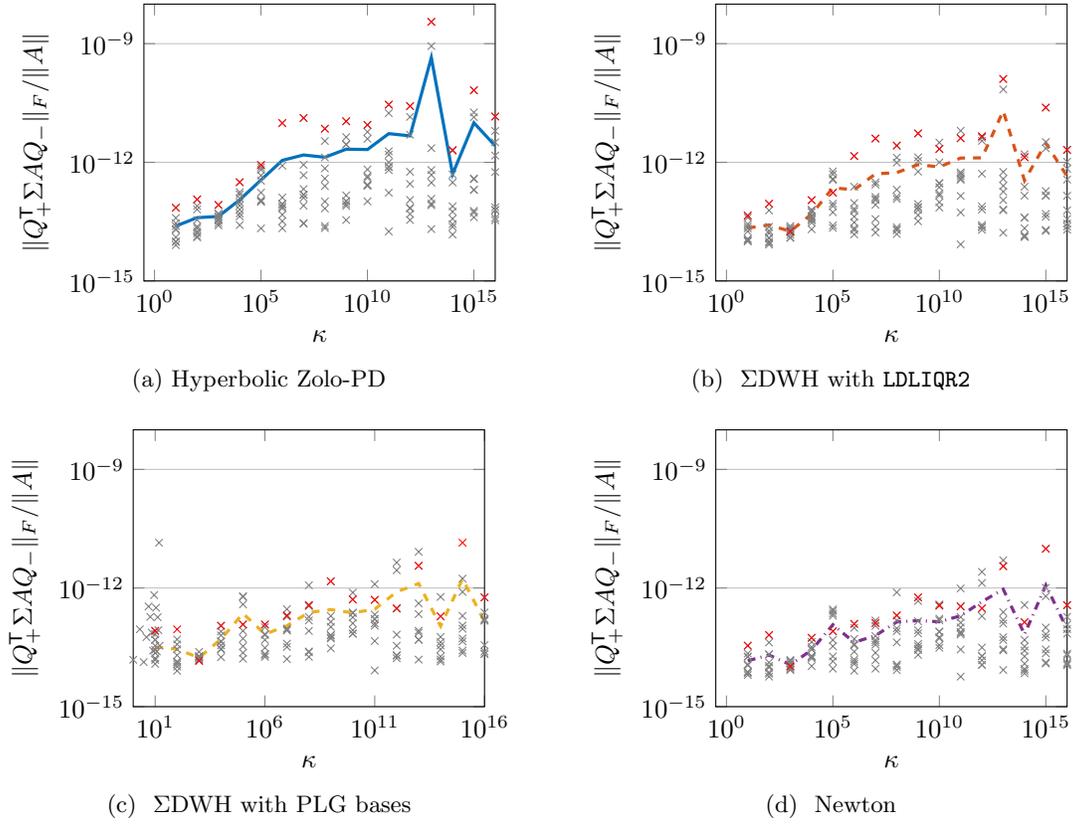              
\medskip
The second example provides first insights on the performance which can be expected by using different methods.
\paragraph{Example 2}
A random matrix of size $5\,000\times 5\,000$ is generated as in Example 1. We measure the number of iterations and the runtime using different methods to compute the matrix sign function. We measure the runtime of the sequential implementation of Zolo-PD, as well as the runtime resulting from its critical path. This means that we only take the runtime of one of the $r$ independent steps in each iteration, i.e.\ the first lines in iterations \eqref{Eq:LDLZolo} and \eqref{Eq:HRZolo}, into account. The measured runtime reflects a performance which can be achieved when these independent computations are implemented in parallel. We compare it to runtimes achieved by $\Sigma DWH$ based on $LDLIQR2$ and $\Sigma DWH$ based on $LDL^{\tran}$ factorizations \cite{BenP22}, and the Newton iteration \cite{ByeX08}. The computation of PLG bases is not yet suited for large-scale performance-critical algorithms, which is why it is not included in the comparison.  The results are found in Table \ref{Tab:ItRuntime}.
 
\begin{table}[h!]
\centering
\begin{tabular}{>{\raggedright}m{0.15\textwidth}>{\raggedright}m{0.225\textwidth}>{\raggedright}m{0.15\textwidth}>{\raggedright}m{0.15\textwidth}m{0.15\textwidth}}
\toprule
&$\kappa$&$10^2$&$10^{8}$&$10^{12}$\\
\midrule
\multirow{4}{*}{\# iterations}
&Hyperbolic Zolo-PD                &2&2&2  \\
&$\Sigma$DWH with \texttt{LDLIQR2} &5&6&6  \\
&$\Sigma$DWH with $LDL^{\tran}$    &5&6&x  \\
&Newton	                           &7&9&9  \\ \midrule
\multirow{4}{*}{runtime }
&Hyperbolic Zolo-PD  (critical path)              &941.70 (298.66)&1136.91 (255.86)&1240.86 (257.70)  \\      
&$\Sigma$DWH with \texttt{LDLIQR2} &883.79&988.39&1067.43   \\
&$\Sigma$DWH with $LDL^{\tran}$    &281.95&304.27&x   \\
&Newton	   		                     &355.05&379.38&416.19   \\   \midrule
                                     
\multirow{4}{\textwidth}{backward error\newline $\frac{\|Q_+^{\tran}\Sigma A Q_-\|_{F}}{\|A\|_{F}}$}
&Hyperbolic Zolo-PD                &7.42e-14&1.05e-11&1.78e-13\\      
&$\Sigma$DWH with \texttt{LDLIQR2} &7.88e-14&2.81e-12&3.29e-13\\
&$\Sigma$DWH with $LDL^{\tran}$    &8.50e-14&1.38e-11&x\\
&Newton	   	                       &1.70e-13&6.66e-13&1.01e-13\\ \midrule
\bottomrule
\end{tabular}
       \caption[l]{\emph{Example 2:} Number of iterations, runtimes and  error for different methods of spectral division for a matrix of size $5\,000\times 5\,000$. $\Sigma DWH$ with $LDL^{\tran}$ did not converge for matrices with $\kappa=10^{12}$.
       \label{Tab:ItRuntime}}
\end{table}
The methods converge as expected and all except $\Sigma$DWH with $LDL^{\tran}$  show good accuracy. $\Sigma$DWH with $LDL^{\tran}$ is known to be unstable for badly conditioned matrices \cite{BenNP22}. However, if it converges, it is the fastest among the measured methods.  The computational effort of one $\Sigma$DWH iteration based on $LDL^{\tran}$ is comparable to the effort of one Newton iteration that is also based on an $LDL^{\tran}$ factorization. $\Sigma$DWH converges in up to 6 steps, and Newton uses up to 9 steps. If \texttt{LDLIQR2} is employed instead of $LDL^{\tran}$ in $\Sigma$DWH, the computational effort doubles, as a second $LDL^{\tran}$ decomposition is used for ``reorthogonalization''. This makes it slower than the Newton iteration. If the critical path of the hyperbolic Zolo-PD is followed, an even lower runtime can be achieved. It could be accelerated at the cost of stability, when $LDL^{\tran}$ decompositions are used instead of \texttt{LDLIQR2}.

\subsection{Applications in electronic structure computations}\label{Subsec:Appl}
We now apply the developed method to two motivating examples concerning electronic excitations in solids and molecules. 
\paragraph{Example 3}
The \texttt{exciting} package \cite{GulKMetal14,VorACetal19} implements various ab initio methods for computing  excited states of solids or molecules, based on (linearized) augmented planewave + local orbital ((L)APW+lo) methods. It can be used to compute the optical scattering spectrum of Lithium Fluoride based on the Bethe-Salpeter equation. The main computational effort in this example is to compute eigenvalues and eigenvectors of a matrix of the form
\begin{align}
 H_{LF} = \begin{bmatrix}
 A_{LF} & B_{LF} \\
 -B_{LF} & -A_{LF}
\end{bmatrix}\in\mathbb{C}^{2560\times 2560},\ A_{LF}=A_{LF}^{\herm},\  B_{LF}=B_{LF}^{\herm}.
\end{align}

$H_{LF}$ is obviously pseudo-Hermitian with respect to $\Sigma=\diag{I_n,-I_n}$. Due to the additional structure, the eigenvalues are known to come in pairs of $\pm \lambda$ \cite{BenP22}. One step of spectral division results in a positive definite matrix, from which all eigenvalues and eigenvectors can be reconstructed. We extracted the matrix from the FORTRAN-based \texttt{exciting} code as a test example for our MATLAB-based prototype.
\begin{table}[h!]
\centering
\begin{tabular}{>{\raggedright}m{0.175\textwidth}>{\raggedright}m{0.15\textwidth}>{\raggedright}m{0.15\textwidth}>{\raggedright}m{0.15\textwidth}m{0.15\textwidth}}
\toprule
&Hyperbolic\newline Zolo-PD&$\Sigma$DWH with \texttt{LDLIQR2}& $\Sigma$DWH with $LDL^{\tran}$& Newton\\
\midrule
\# iterations &2&5&5&7\\
Zolotarev rank &4&1 & 1 & not applicable\\
backward error (Chol, Alg. \ref{Alg:HypSubspaceChol})  &1.02e-10&7.42e-11&9.62e-11&2.48e-10\\ 
backward error (LDL, Alg. \ref{Alg:HypSubspaceDefLDL}) &6.93e-18&7.26e-18&6.99e-18&1.48e-17\\ 
\bottomrule
\end{tabular}
       \caption[l]{\emph{Example 3:} Results for Bethe-Salpeter matrix computed for Lithium Fluoride.
       \label{Tab:ItResLF}}
\end{table}

The results in Table \ref{Tab:ItResLF} show that convergence is achieved in a limited number of iterations for all methods, as expected. The reported backward error $\frac{\|Q_+^{\tran}\Sigma A Q_-\|_{F}}{\|A\|_{F}}$ depends largely on the chosen method for computing a hyperbolic subspace representation. The Cholesky-based method does not work well. The eigenvalues smallest in modulus easily ``pass over'', such that the computed quantities $\Sigma P_+$ or $-\Sigma P_-$ have negative eigenvalues. The Cholesky-based method in Algorithm \ref{Alg:HypSubspaceChol} does not accurately capture this behavior, while the $LDL^{\tran}$-based method alleviates the effect through pivoting. 

All methods for computing the matrix sign function work equally well concerning accuracy because $H_{LF}$ is well conditioned ($\cond{H_{LF}}\approx 10$).

\paragraph{Example 4}
In \cite{BenKK16,BenDKetal17} a Bethe-Salpeter approach is explored in the context of tensor-structured Hartree-Fock theory for molecules \cite{RebTS13}. We consider the $N_2H_4$ example in \cite{BenKK16}. With real-valued orbitals the derivation arrives at a structured eigenvalue problem similar to Example 3, but with real values. 
\begin{align}
 H_{N_2H_4} = \begin{bmatrix}
 A_{N_2H_4} & B_{N_2H_4} \\
 -B_{N_2H_4}^{\tran} & -A_{N_2H_4}^{\tran}
\end{bmatrix}\in\mathbb{R}^{1314\times 1314},\ A_{N_2H_4}=A_{N_2H_4}^{\tran},\  B_{N_2H_4}\approx B_{N_2H_4}^{\tran}.
\end{align}
While the original derivation in \cite{RebTS13} yields a symmetric off-diagonal block $B$, in the construction in \cite{BenKK16}, this property is lost. The property of pseudosymmetry, however, is not affected, making our developed method applicable.

\begin{table}[h!]
\centering
\begin{tabular}{>{\raggedright}m{0.175\textwidth}>{\raggedright}m{0.15\textwidth}>{\raggedright}m{0.15\textwidth}>{\raggedright}m{0.15\textwidth}m{0.15\textwidth}}
\toprule
&Hyperbolic\newline Zolo-PD&$\Sigma$DWH with \texttt{LDLIQR2}& $\Sigma$DWH with $LDL^{\tran}$& Newton\\
\midrule
\# iterations &2&5&5&7\\
Zolotarev rank &5&1 & 1 & not applicable\\
backward error (Chol, Alg. \ref{Alg:HypSubspaceChol})  &1.19e-18&9.23e-19&1.46e-17&2.04e-18\\ 
backward error (LDL, Alg. \ref{Alg:HypSubspaceDefLDL}) &1.22e-18&9.62e-19&1.46e-17&2.13e-18\\ 
\bottomrule
\end{tabular}
       \caption[l]{\emph{Example 4:} Results for Bethe-Salpeter matrix computed for $N_2H_4$.
       \label{Tab:ItResN2H4}}
\end{table}

Numerical results of the spectral division are found in Table \ref{Tab:ItResN2H4}. All methods yield good results ($\cond{H_{N_2H_4}}\approx 5$). In contrast to Example 3, no problem occurs when the Cholesky decomposition is used for computing hyperbolic subspace representations. An explanation is probably linked to the fact that real matrices instead of complex ones are considered but requires further investigation.

\begin{figure}[ht]
\centering
\setlength\figureheight{0.5\textwidth}
\setlength\figurewidth{0.7\textwidth}
%
%
\definecolor{mycolor1}{rgb}{0.00000,0.44700,0.74100}%
\definecolor{mycolor2}{rgb}{0.85000,0.32500,0.09800}%
\begin{tikzpicture}

\begin{axis}[%
width=0.951\figurewidth,
height=\figureheight,
at={(0\figurewidth,0\figureheight)},
scale only axis,
xmin=0,
xmax=30,
ymin=6,
ymax=16,
grid=both,
axis background/.style={fill=white},
xlabel = eigenvalues,
ylabel = energy (eV),
legend style={cells={align=center}, font = \small, draw=gray},
legend pos=south east
]

\addplot [color=mycolor1, mark=square, mark options={solid, mycolor1}, very thick]
  table[row sep=crcr]{%
1	7.07309980765472\\
2	7.08653473796964\\
3	7.9276276455983\\
4	7.95420169878752\\
5	8.72777308687677\\
6	8.74217818613574\\
7	9.53198176222625\\
8	9.54118062960048\\
9	9.79913016276075\\
10	9.81343896346495\\
11	12.2577986324441\\
12	12.2913207163301\\
13	12.9193556142916\\
14	12.9474492307876\\
15	12.9564265594342\\
16	12.961949533496\\
17	13.0730618159105\\
18	13.0946213498661\\
19	13.6461891096159\\
20	13.678293729102\\
21	13.7195156113558\\
22	13.7458924832696\\
23	14.419833128136\\
24	14.4246315547007\\
25	14.77101172893\\
26	14.7958418545424\\
27	15.1161790523069\\
28	15.1199732156908\\
29	15.2514552784035\\
30	15.2557554267822\\
};
\addlegendentry{Original matrix $H_{N_2H_4}\in\mathbb{R}^{2n\times 2n}$}

\addplot [color=mycolor2, only marks, mark=x, mark options={solid, mycolor2, very thick, scale =2}]
  table[row sep=crcr]{%
1	7.0730998076542\\
3	7.92762764559616\\
5	8.72777308687579\\
7	9.53198176222612\\
9	9.79913016276224\\
11	12.257798632444\\
13	12.9193556142909\\
15	12.956426559433\\
17	13.0730618159098\\
19	13.6461891096148\\
21	13.7195156113553\\
23	14.4198331281358\\
25	14.7710117289294\\
27	15.1199732156912\\
29	15.2514552784031\\
};
\addlegendentry{Positve definite matrix $H_{N_2H_4,+}\in\mathbb{R}^{n\times n}$\\ after spectral division}

\end{axis}
\end{tikzpicture}%
 \caption{\emph{Example 4}: Absolute values of eigenvalues corresponding to $N_2H_4$. }\label{Fig:KhoromskiEVs}
\end{figure}
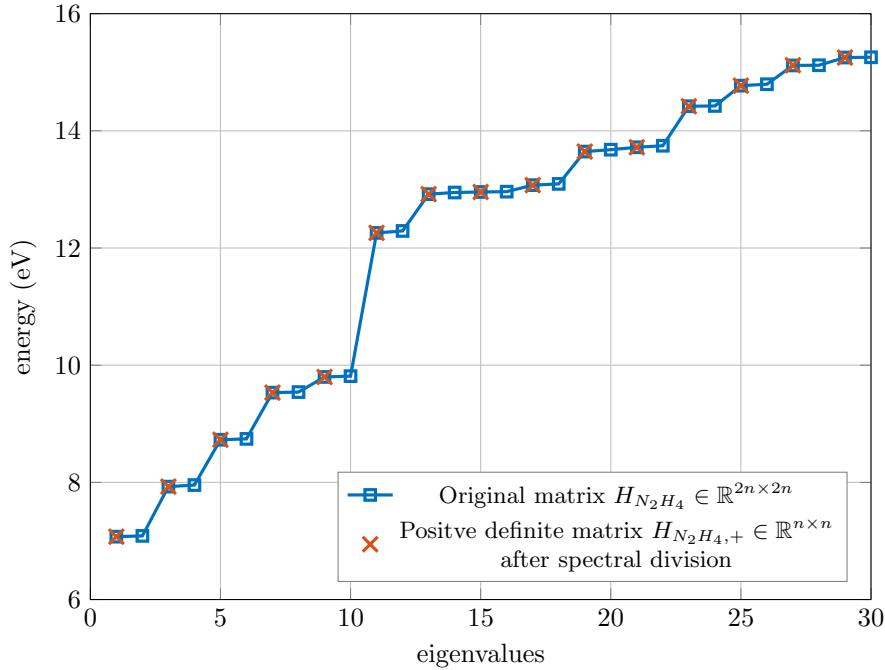

Figure \ref{Fig:KhoromskiEVs} corresponds to Figure 2 in \cite{BenKK16} and displays absolute values of the eigenvalues of $H_{N_2H_4}$. The red crosses denote the eigenvalues of the positive definite matrix resulting after one step of spectral division ($A_{11}$ in Algorithm \ref{Alg:PseuDivConq}). The remaining eigenvalues have (approximately) equal modulus, but opposite sign and are found as the eigenvalues of the negative definite matrix ($A_{22}$ in Algorithm \ref{Alg:PseuDivConq}).

\section{Conclusions}\label{Sec:Conclusions}
We presented a generalization of the well-known spectral divide-and-conquer approach for the computation of eigenvalues and eigenvectors of pseudosymmetric matrices. In particular, when matrices with additional definiteness properties are considered, many parallels to the symmetric divide-and-conquer method become apparent. These parallels allow a computation of the matrix sign function, the key element for spectral division approaches, in just two iterations, using Zolotarev functions. Furthermore, the eigenvalue problem is decoupled into two smaller symmetric eigenvalue problems that can be solved with existing techniques. The presented algorithm is a promising new approach in the field of computing electronic excitations. 

As we presented a completely new approach for structured eigenvalue computations, naturally, many possible future research directions open up as a consequence of this work. It is possible to use permuted Lagrangian graph bases, as presented in \cite{BenP22}, to further improve the accuracy of the Zolotarev iteration for computing the matrix sign function. This should go hand in hand with a well-founded analysis of the stability of the proposed methods. In the same vein, the numerical behavior of the subspace computations (Algorithms \ref{Alg:HypSubspaceChol} and \ref{Alg:HypSubspaceDefLDL}) is not yet fully understood, as the examples presented in Section \ref{Subsec:Appl} show. Regarding the applications concerning electron excitation, the matrices (\eqref{Eq:BSE1} to \eqref{Eq:BSE2}) show even more structure than has been exploited in the presented methods. Making the proposed iterations aware of these structures, such that they operate directly on the matrix blocks $A$ and $B$, is a promising direction towards even more efficient methods.




\addcontentsline{toc}{section}{References}

\end{document}